\documentclass[dvipsnames]{article}

\usepackage{mathtools}
\usepackage{amssymb,amsthm}
\usepackage{graphicx}
\usepackage{xcolor}
\usepackage{pgfplots}
\usepackage{subcaption}
\usepackage{siunitx}

\colorlet{MyBlue}{Blue!50!Cerulean}
\definecolor{Viridian}{RGB}{64,130,109}
\definecolor{Pewter}{RGB}{142,146,148}
\definecolor{Vermilion}{RGB}{217,56,30}
\definecolor{Celadon}{HTML}{ACE1AF}
\definecolor{Cinnabar}{RGB}{227,66,52}
\definecolor{Saffron}{RGB}{244,196,48}
\definecolor{PrussianBlue}{RGB}{0,49,83}
\colorlet{TanbaIro}{PrussianBlue!60!SkyBlue}
\definecolor{Ebony}{RGB}{85,93,80}
\definecolor{SakuraIro}{RGB}{252,201,185}
\definecolor{Azalea}{RGB}{216,69,101}
\definecolor{Teak}{RGB}{171, 137, 83}
\definecolor{Silver}{RGB}{192,192,192}

\pgfplotscreateplotcyclelist{kanto}{%
    {MyBlue,every mark/.append style={solid,fill=.!60!white}, mark=pentagon*},
    {Viridian,every mark/.append style={solid,fill=.!60!white},mark=*},
    {Pewter,every mark/.append style={solid,fill=.!60!white},mark=square*},
    {Cerulean,every mark/.append style={solid,fill=.!60!white},mark=otimes*},
    {Vermilion,every mark/.append style={solid,fill=.!60!white},mark=oplus*},
    {Lavender,every mark/.append style={solid,fill=.!60!white},mark=diamond*},
    {Celadon,every mark/.append style={solid,fill=.!60!white},mark=triangle*},
    {Fuchsia,every mark/.append style={solid,fill=.!60!white},mark=star},
    {Saffron,mark=asterisk},
    {Cinnabar,every mark/.append style={solid,fill=.!60!white},mark=+},
}%

\pgfplotscreateplotcyclelist{johto}{
    {SpringGreen,every mark/.append style={solid,fill=.!70!white},mark=*},
    {SakuraIro,every mark/.append style={solid,fill=.!70!white},mark=diamond*},
    {Violet,every mark/.append style={solid,fill=.!70!white},mark=square*},
    {Azalea,every mark/.append style={solid,fill=.!70!white},mark=triangle*},
    {Goldenrod,every mark/.append style={solid,fill=.!70!white},mark=pentagon*},
    {Teak,every mark/.append style={solid,fill=.!70!white},mark=+},
    {OliveGreen, every mark/.append style={solid,fill=.!70!white}, mark=star},
    {TanbaIro, every mark/.append style={solid,fill=.!70!white}, mark=asterisk},
    {Mahogany, every mark/.append style={solid,fill=.!70!white}},
    {Ebony, every mark/.append style={solid,fill=.!70!white}},
    {Silver, every mark/.append style={solid,fill=.!70!white}}
}%

\pgfplotscreateplotcyclelist{adaptivity}{%
    {Cerulean,every mark/.append style={solid,fill=.!40!white},mark=otimes*},
    {LimeGreen,every mark/.append style={solid,fill=.!40!white},mark=triangle*},
    {BurntOrange,every mark/.append style={solid,fill=.!40!white},mark=oplus*},
    {Cerulean,every mark/.append style={solid,fill=.!40!white},mark=x},
    {LimeGreen,every mark/.append style={solid,fill=.!40!white},mark=triangle},    
    {BurntOrange,every mark/.append style={solid,fill=.!40!white},mark=+},    
    {Fuchsia,every mark/.append style={solid,fill=.!40!white},mark=pentagon*},
    {Fuchsia,every mark/.append style={solid,fill=.!40!white},mark=pentagon},
    {OliveGreen,every mark/.append style={solid,fill=.!40!white},mark=diamond*},
    {OliveGreen,every mark/.append style={solid,fill=.!40!white},mark=diamond},
}%
\pgfplotscreateplotcyclelist{adaptivity5}{%
    {Cerulean,mark=*},
    {Fuchsia,mark=otimes*},
    {BurntOrange,mark=oplus*},
    {LimeGreen,mark=triangle*},
    {OliveGreen,mark=diamond*},
    {Cerulean,mark=o},
    {Fuchsia,mark=otimes},
    {BurntOrange,mark=oplus},    
    {LimeGreen,mark=triangle},    
    {OliveGreen,mark=diamond},
}%

\pgfplotsset{
    every axis/.append style={
        grid=major,%
        grid style={densely dotted}
    },%
    every axis legend/.append style={
        font=\tiny,%
        cells={anchor=west}%
    },%
    every axis plot/.append style={%
        thick
    },%
    log x ticks with fixed point/.style={
        xticklabel={
            \pgfkeys{/pgf/fpu=true}
            \pgfmathparse{exp(\tick)}%
            \pgfmathprintnumber[fixed relative, precision=3]{\pgfmathresult}
            \pgfkeys{/pgf/fpu=false}
        }
    },
    log y ticks with fixed point/.style={
        yticklabel={
            \pgfkeys{/pgf/fpu=true}
            \pgfmathparse{exp(\tick)}%
            \pgfmathprintnumber[fixed relative, precision=3]{\pgfmathresult}
            \pgfkeys{/pgf/fpu=false}
        }
    },
    log y ticks with percent/.style={
        yticklabel={
            \pgfkeys{/pgf/fpu=true}
            \pgfmathparse{exp(\tick)}%
            \pgfmathparse{100*\pgfmathresult}
            $ \pgfmathprintnumber[fixed, precision=2]{\pgfmathresult} \%$
            \pgfkeys{/pgf/fpu=false}
        }
    },
}

\title{Adaptive space-time finite element methods\\ for 
non-autonomous parabolic 
problems\\ with distributional sources}
\author{Ulrich Langer, Andreas Schafelner}

\newtheorem{theorem}{Theorem}
\newtheorem{lemma}[theorem]{Lemma}
\theoremstyle{definition}
\newtheorem{remark}[theorem]{Remark}

\definecolor{ao(english)}{rgb}{0.0, 0.5, 0.0}

\begin{document}

\maketitle

\begin{abstract}
We consider locally stabilized, conforming finite element schemes on  completely unstructured simplicial space-time meshes 
for the numerical solution of parabolic initial-boundary value problems 
with variable, possibly discontinuous in space and time, coefficients.
Distributional sources are also admitted.
Discontinuous coefficients, non-smooth boundaries, changing boundary conditions,
non-smooth or incompatible initial conditions, and non-smooth right-hand sides 
can lead to non-smooth solutions.\\
We present new a priori and a posteriori error estimates for low-regularity solutions.
In order to avoid reduced convergence rates appearing in the case of uniform mesh refinement,
we also consider adaptive refinement procedures based on residual
a posteriori error indicators and functional a posteriori error estimators.
The huge system of space-time finite element equations is then solved by means of
GMRES preconditioned by space-time algebraic multigrid. 
In particular,  in the 4d space-time case that is 3d in space, simultaneous space-time parallelization can considerably reduce the computational time. 
We present and discuss numerical results for several examples possessing
different regularity features.
\\[1em]
\end{abstract}

\section{Introduction}
\label{LS:Section:Introduction}
Time-parallel and space-time methods, in particular, space-time finite element methods have a long history. 
We refer the interested reader to the survey articles 
\cite{LS:Gander2015a} and \cite{LS:SteinbachYang:2019a}
for a comprehensive review of time-parallel and space-time methods, respectively.
We here only mention the Streamline Upwind Petrov – Galerkin (SUPG) finite element method (sometimes also called Streamline Diffusion (SD) method),
that was originally proposed by Hughes and Brooks for 
solving stationary convection-dominated convection-diffusion problems 
\cite{LS:HughesBrooks:1979a}, {and}
was {later} used by Johnson and Saranen \cite{LS:JohnsonSaranen:1986a} 
to construct stable finite element schemes 
for time-dependent 
convection-diffusion
problems, discretizing space and time simultaneously in  time-slices.
Similarly, in \cite{LS:HughesFrancaHulbert:1989a}, Hughes, Franca and Hulbert proposed and analyzed the 
more general Galerkin/Least-Squares Finite Element Method (FEM) for stationary and instationary problems
including time-dependent convection-diffusion problems. 
Both techniques were then used in many publications to stabilize
finite element schemes.
The revival of space-time methods is certainly connected 
with the availability of massively parallel computers 
with thousands or even millions of cores, but also with the simultaneous adaptivity in space and time,
the easy handling of moving spatial domains and interfaces that are fixed in the space-time domain $Q$,
and the all-at-once treatment of first-order optimality conditions (KKT system) in optimal control 
problems with parabolic or hyperbolic PDE constraints.
Recently, time-upwind stabilizations have been used to construct and analyze
coercive space-time Isogeometric Analysis (IGA) schemes in \cite{LS:LangerMooreNeumueller:2016a},
partial low-rank tensor IGA schemes  
\cite{LS:MantzaflarisScholzToulopoulos:2019a}, 
and
space-time $hp$-finite element schemes combined with mesh-grading in time to treat singularities in time \cite{LS:DevaudSchwab:2018a}.
These papers rely on a tensor-product structure of space and time. Results related to methods that do not assume a tensor-product structure were presented, e.g., in \cite{LS:Steinbach:2015a}, where the author proposes and analyzes a $\inf$-$\sup$-stable space-time FEM on completely unstructured decompositions of the space-time domain, and in \cite{LS:BankVassilevskiZikatanov:2016a}, where an arbitrary dimension convection-diffusion scheme for space-time problems is proposed, with two different discretizations. 
Unstructured space-time finite element and IGA methods have also been used to solve involved 
engineering problems; see, e.g., \cite{LS:Behr2008a,LS:KaryofylliWendlingMakeHostersBehr:2019a}, 
and the references therein, where mostly unstructured meshes are used within time slabs 
into which the space-time cylinder $Q$ is decomposed.
Following our previous work \cite{LS:LangerNeumuellerSchafelner:2019a},
we will derive, analyze and test space-time finite element methods on completely 
unstructured simplicial space-time meshes 
for non-autonomous parabolic initial-boundary value problems with 
distributional right-hand sides of the form: find  $ u $ such that
\begin{align}
\partial_{t}u - \mathrm{div}_x(\nu\, \nabla_{x} u ) &= f - \,\mathrm{div}_x(\mathbf{f})&&\text{in}\ Q=\Omega\times(0,T), \label{LS:eq:modelproblem}\\
u&=u_D := 0&&\text{on}\ \Sigma = \partial\Omega\times(0,T),\label{LS:eq:modelproblem:bc}\\
u&=u_0 &&\text{on}\ \Sigma_0 = \Omega\times\{0\},\label{LS:eq:modelproblem:ic}
\end{align}
where the spatial domain $\Omega \subset \mathbb{R}^d$, $d=1,2,3$, is bounded and Lipschitz,
$T > 0$ denotes the final time, 
$\nu \in L_\infty(Q)$ is a given uniformly bounded and positive coefficient
that may discontinuously depend on the spatial variable $x =(x_1,\ldots,x_d)$ and the time variable $t$,
$ f \in L_2(Q) $, and $\mathbf{f} \in [L_2(Q)]^d$.
If $\mathbf{f} = \mathbf{0}$ and $\nu$ is of bounded variation in $t$ for almost 
all $x$, then the weak solution $u$ of (\ref{LS:eq:modelproblem})--(\ref{LS:eq:modelproblem:ic})
has maximal parabolic regularity, i.e., 
$\partial_{t}u \in L_2(Q)$ and $\mathcal{L}_x u := - \mathrm{div}_x(\nu\, \nabla_{x} u )\in L_2(Q)$,
see \cite{LS:Dier:2015a}. Thus, in the maximal parabolic regularity setting, the PDE 
$\partial_{t}u - \mathrm{div}_x(\nu\, \nabla_{x} u ) = f$ makes sense in $L_2(Q)$,
and this setting was the starting point for deriving the consistent space-time 
finite element schemes in \cite{LS:LangerNeumuellerSchafelner:2019a}.
If $\mathbf{f} \neq \mathbf{0}$ and it is not smooth enough, i.e., $\mathbf{f}$ creates a distribution that is not in $L_2(Q)$, 
then maximal parabolic regularity does not hold. However, in many applications, the vector field $\mathbf{f}$ has a special
structure. For instance, in 2d eddy current problems, $\nu$ denotes the reluctivity, 
$f$ is nothing but the impressed current in the coils, and $\mathbf{f}$ represents 
the magnetization in the permanent magnets. In such a way, we can model a rotating 
electrical machine, see, e.g., \cite{LS:PyrhoenenJokinenHrabovcova:2008a}.
Motivated by such applications, we can assume that there is a decomposition of 
the space-time cylinder $ \overline{Q} = \bigcup_{i=1}^m \overline{Q_i} $ 
into $m$ non-overlapping, sufficiently smooth subdomains $Q_1,\ldots,Q_m$ such that 
$\mathbf{f}$ is sufficiently smooth in $Q_i$ but may have jumps across the 
boundaries of these subdomains. These jumps create distributions located on the 
interfaces between the subdomains. 
In this paper, we consider exactly this case of such a special right-hand side with a piecewise 
smooth vector-field $\mathbf{f}$ leading to a distribution that is located on the interfaces.
We again derive a consistent space-time finite element scheme, for which we prove 
a best-approximation (C\'{e}a-like) estimate that immediately yields convergence rate 
estimates in terms of the mesh-size $h$ under additional regularity assumptions.
In the case of low-regularity solutions, say, $u \in H^{1+s}(Q)$ with some $s \in (0,1]$,
and uniform mesh refinement,
we can only expect a convergence rate of $O(h^s)$ in the corresponding energy norm
independent of the polynomial order $p$ of the finite element shape functions used.
To circumvent this loss in the convergence rate, we also consider full space-time adaptive 
finite element procedures based on residual indicators or functional error estimators.
We compare the efficiency of these adaptive approaches in a series of numerical 
experiments.

The remainder of the paper is organized as follows. 
In Section~\ref{LS:Section:SpaceTimeVFandSolvabilityResults}, we recall 
the standard space-time variational formulation 
and solvability results including maximal parabolic regularity 
and other regularity results.
Section~\ref{LS:Section:SpaceTimeFEM} is devoted to the derivation 
of consistent space-time finite element schemes. 
New a priori error estimates are presented in Section~\ref{LS:Section:Apriori}.
Section~\ref{LS:Section:Aposteriori} deals with a posteriori error estimates, 
which are used to control the 
adaptive finite element procedures in our numerical experiments.
In Section~\ref{LS:Section:NumericalResults}, we present various numerical results 
for uniform and adaptive mesh refinement in space and time.
Finally, we draw some conclusions in Section~\ref{LS:Section:Conclusions}.

\section{Space-time variational formulation and solvability results}
\label{LS:Section:SpaceTimeVFandSolvabilityResults}

We start with Ladyzhenskaya's space-time variational formulation of the para\-bo\-lic
initial-boundary value problem (\ref{LS:eq:modelproblem})--(\ref{LS:eq:modelproblem:ic}) 
in space-time Sobolev spaces, see \cite{LS:Ladyzhenskaya:1954a} 
and the classical monographs \cite{LS:LadyzhenskayaetSolonnikovUraltseva:1967a,LS:Ladyzhenskaya:1985a}:
find $ u\in H^{1,0}_0(Q):= \{v \in L_2(Q):\, \nabla_x u \in [L_2(Q)]^d,\, u=0\;\mbox{on}\;\Sigma\}$ such that
\begin{equation}
\label{LS:eq:SpaceTimeVF}
\begin{aligned}
    a(u,v):=\int_{Q} (-u \partial_{t}v + \nu \nabla_{x}u\cdot\nabla_{x}v) \;\mathrm{d}Q = \\
    \int_{Q} (f\,v + \mathbf{f}\cdot\nabla_{x}v)\;\mathrm{d}Q + \int_\Omega u_0(x)\,v(x,0)\;\mathrm{d}x
    =: l(v),\ \forall v \in H^{1,1}_{0,\overline{0}}(Q),
\end{aligned}
\end{equation}
where 
$H^{1,1}_{0,\overline{0}}(Q) := \{v \in L_2(Q):\, \nabla_x u \in [L_2(Q)]^d,\, \partial_t u \in L_2(Q),\, 
u=0\;\mbox{on}\;\Sigma,\, u=0\;\mbox{on}\;\Sigma_T = \Omega \times \{T\} \}$.
If 
$f \in L_{2,1}(Q):=\{ v: Q \rightarrow \mathbb{R} : \int_{0}^{T}\! \| v(\cdot,t) \|_{L_2(\Omega)} \; \mathrm{d}t < \infty \}$, $\mathbf{f} \in [L_2(Q)]^d$, $u_0 \in L_2(\Omega)$, 
and 
the coefficient $\nu \in L_\infty(Q)$
is uniformly bounded and positive, i.e., there exist positive constants 
$\underline{\nu}$ and $\overline{\nu}$ such that
\begin{equation}
 \label{LS:eq:nu}
 \underline{\nu} \le \nu(x,t) \le  \overline{\nu}\; \mbox{for almost all}\; (x,t) \in Q,
\end{equation}
then the space-time variational formulation (\ref{LS:eq:SpaceTimeVF}) has a unique solution $u$ that even belongs 
to the space $\mathring{V}_2^{1,0}(Q) \subset H^{1,0}_0(Q)$; 
see \cite{LS:Ladyzhenskaya:1985a}, Chapter III, Paragraph 3, Theorem~3.2.
In particular, this means that $\|u(\cdot,t)\|_{L_2(\Omega)}$ is continuous with respect to (wrt) $t$.
Therefore, the initial condition makes sense in $L_2(\Omega)$.
The same result can be obtained if we consider the solvability 
of (\ref{LS:eq:modelproblem})--(\ref{LS:eq:modelproblem:ic}) 
in the Bochner space 
$W(0,T) := \{v \in L_2(0,t;H_0^1(\Omega)): \partial_t u \in L_2(0,t;H^{-1}(\Omega))\}$ 
of abstract functions $u$ mapping the time interval $[0,T]$ to some Hilbert or Banach space $X(\Omega)$. Indeed, 
the parabolic initial-boundary value problem is uniquely solvable in $W(0,T)$, 
and $W(0,T)$ is continuously embedded into $C([0,T],L_2(\Omega))$;
see, e.g., \cite{LS:Lions:1971}.
This view at parabolic initial-boundary value  is closely connected with the 
so-called line variational formulations and the corresponding operator 
ordinary differential equations, and is the starting point for time-stepping methods.

If $f \in L_2(Q)$, $\mathbf{f} = \mathbf{0}$, $u_0 \in H_0^1(\Omega)$, and $\nu=1$,
then the  solution  $ u\in {H}^{1,0}_{0}(Q) $ of (\ref{LS:eq:SpaceTimeVF})  belongs to the space
$H^{\Delta,1}_0(Q) = \{v \in H^1_0(Q) : \Delta_x v \in L_2(Q) \},$
and continuously depends on $t$ in the norm of the space $H^1_0(\Omega)$;
see \cite{LS:Ladyzhenskaya:1985a}, Chapter III, Paragraph 2, Theorem~2.1.
This property is called maximal parabolic regularity.
If $\nu = \nu(x,t)$ depends on $x$ and $t$ (non-autonomous case) possibly in discontinuous way
(like in electrical machines, moving interfaces), then 
it is more involved to specify the most general conditions for $\nu$ such 
that maximal parabolic regularity happens, i.e. the solution 
$ u\in {H}^{1,0}_{0}(Q) $ of (\ref{LS:eq:SpaceTimeVF})
belongs to the space
$$H^{\mathcal{L},1}_0(Q) = \{v \in H^1_0(Q) : \mathcal{L}_x v :=  \mbox{div}_x(\nu \nabla_x  u) \in L_2(Q) \}.$$
Recently, Dier \cite{LS:Dier:2015a}
proved maximal parabolic regularity provided that the corresponding 
elliptic form is of bounded variation. In our case, this means that $\nu$ is of bounded variation in $t$ for almost 
all $x$.
In the case of maximal parabolic regularity, we have $\partial_t u -\mathcal{L}_x u = f$ in $L_2(Q)$ 
that was the starting point for the derivation of consistent space-time finite element schemes 
in \cite{LS:LangerNeumuellerSchafelner:2019a}.

In general, for non-trivial $\mathbf{f} \in [L_2(Q)]^d$ (with jumps), we cannot expect maximal parabolic regularity.
Motivated by applications, we assume that there exist  a decomposition of 
the space-time cylinder $ \overline{Q} = \bigcup_{i=1}^m \overline{Q_i} $ 
into $m$ non-overlapping Lipschitz domains $Q_1,\ldots,Q_m$ such that 
the restriction of $\mathbf{f}$ to $Q_i$ belongs to $H(\mbox{div}_x,Q_i)$ for all $i=1,2,\ldots,m$, 
but may have jumps across the boundaries of these subdomains.
Now, taking an arbitrary test function $v \in H^{1,1}_{0,\overline{0}}(Q)$ with a compact support in $Q_i$,
we immediately  get
\begin{equation}
\label{LS:eq:STVFinQi}
 \int_{Q_i} (-u \partial_{t}v + \nu \nabla_{x}u\cdot\nabla_{x}v) \;\mathrm{d}Q 
 = \int_{Q_i} (f\,v + \mathbf{f}\cdot\nabla_{x} v)\;\mathrm{d}Q.
\end{equation}
from the space-time variational formulation (\ref{LS:eq:SpaceTimeVF}).
Rewriting the left-hand side and applying integration by parts on the right-hand side 
in (\ref{LS:eq:STVFinQi}), 
we obtain
\begin{equation}
\int_{Q_i} \begin{pmatrix}
\nu\nabla_{x}u\\ -u
\end{pmatrix}\cdot\nabla v \;\mathrm{d}Q = \int_{Q_i} (f - \mathrm{div}_x \mathbf{f})v\;\mathrm{d}Q,
\end{equation}
where $\nabla = \begin{pmatrix}\nabla_x^\top & \partial_t \end{pmatrix}^\top$ is the space-time gradient.
This is nothing but the definition of the weak space-time divergence, i.e., we have 
\begin{equation}
\label{LS:eq:div}
    \mathrm{div}\begin{pmatrix} -\nu\nabla_{x}u\\u\end{pmatrix} = f - \mathrm{div}_x\mathbf{f}\quad\text{in}\ L_2(Q_i).
\end{equation}
The solution
$u \in \mathring{V}_2^{1,0}(Q)$ of (\ref{LS:eq:SpaceTimeVF})
also satisfies the integral identity; 
see \cite{LS:Ladyzhenskaya:1985a}, Chapter~3, identity (3.20) on p. 120,
\begin{equation}
\label{LS:eq:SpaceTimeVFinV10}
\begin{aligned}
\int_{Q} (-u \partial_{t}v + \nu \nabla_{x}u\cdot\nabla_{x}v) \;\mathrm{d}Q 
+ \int_\Omega u(x,T)\,v(x,T)\;\mathrm{d}x= \\
    \int_{Q} (f\,v + \mathbf{f}\cdot\nabla_{x}v)\;\mathrm{d}Q + \int_\Omega u_0(x)\,v(x,0)\;\mathrm{d}x
   ,\ \forall v \in H^{1}_{0}(Q),
\end{aligned}
\end{equation}
from which one can derive the space-time flux relation
\begin{equation}
\label{LS:eq:SpaceTimeFluxH1}
\begin{aligned}
    \sum_{i=1}^m \int_{\partial Q_i}\!\begin{pmatrix}\nu\nabla_{x}u\\-u\end{pmatrix}\cdot\begin{pmatrix} \mathbf{n}_x\\{n}_t\end{pmatrix}v\;\mathrm{d}s + \int_\Omega u(x,T)\,v(x,T)\;\mathrm{d}x\\
    = 
    \sum_{i=1}^m \int_{\partial Q_i}\! \mathbf{f}\cdot\mathbf{n}_x\, v\;\mathrm{d}s + \int_\Omega u_0\,v(x,0)\;\mathrm{d}x,\;\forall v \in H^{1}_{0}(Q).
\end{aligned}    
\end{equation}
Indeed, in (\ref{LS:eq:SpaceTimeVFinV10}), we can rewrite the integrals over $Q$ as sum of 
integrals over $Q_i$. Then we have the variational identity
\begin{align*}
\sum_{i=1}^m \int_{Q_i}\! \begin{pmatrix}\nu\nabla_{x}u\\ -u\end{pmatrix}\cdot\nabla v \;\mathrm{d}Q + \int_\Omega u(x,T)\,v(x,T)\;\mathrm{d}x\\
=
\sum_{i=1}^m \int_{Q_i}\! (f\,v + \mathbf{f}\cdot\nabla_{x}v)\;\mathrm{d}Q 
+ \int_\Omega u_0\,v(x,0)\;\mathrm{d}x
\end{align*}
for all $v\in H^1_{0}(Q) = H^{1,1}_{0}(Q)$.
Integrating by parts in the left-hand side and in the right-hand side gives the identity
\begin{align}
\label{LS:eq:VariationalIdentity}
\nonumber
    \sum_{i=1}^m \left(\int_{Q_i}\!\mathrm{div}\begin{pmatrix}-\nu\nabla_{x}u\\u\end{pmatrix} v\;\mathrm{d}Q
    + \int_{\partial Q_i}\!\begin{pmatrix}\nu\nabla_{x}u\\-u\end{pmatrix}\cdot\begin{pmatrix} \mathbf{n}_x\\{n}_t\end{pmatrix}v\;\mathrm{d}s\right)\\ \nonumber
    + \int_\Omega u(x,T)\,v(x,T)\;\mathrm{d}x\\
    = 
    \sum_{i=1}^m\left(\int_{Q_i}\! (f-\mathrm{div}_x \mathbf{f})v\;\mathrm{d}Q + \int_{\partial Q_i}\! \mathbf{f}\cdot\mathbf{n}_x\, v\;\mathrm{d}s\right) + \int_\Omega u_0\,v(x,0)\;\mathrm{d}x,
\end{align}
that holds for all $v\in H^{1}_{0}(Q)$. We mention that the integral over $\partial Q_i$ have to be 
understood as duality products on $H^{-1/2}(\partial Q_i) \times H^{1/2}(\partial Q_i)$.
Using (\ref{LS:eq:div}) in (\ref{LS:eq:VariationalIdentity}), we arrive at the space-time flux 
relation (\ref{LS:eq:SpaceTimeFluxH1}).
%

\section{Space-time finite element methods}
\label{LS:Section:SpaceTimeFEM}

Next, we need a decomposition (triangulation) $ \mathcal{T}_h $ of $ Q $ into shape-regular finite elements $K$ 
such that $\overline{Q} = \bigcup_{K\in\mathcal{T}_h} \overline{K}$ and $K\cap K'=\emptyset$ 
for all $K$ and $K'$ from $ \mathcal{T}_h $ with $K \neq K'$; see \cite{LS:Braess:2013a,LS:BrennerScott:2008a,LS:Ciarlet:1978a} 
for a precise definition of shape-regular triangulations.
We assume that the triangulation is aligned with $Q_i$, i.e. 
$\overline{Q}_i = \cup_{K \in \mathcal{T}_{i,h}}\overline{K} $ can be represented by a subset 
$ \mathcal{T}_{i,h}$ of the triangulation $ \mathcal{T}_{h} = \cup_{i=1}^m \mathcal{T}_{i,h}$.

On the basis of the shape-regular triangulation $ \mathcal{T}_h $, we define the space-time finite element space
\begin{equation}
\label{LS:eq:V0h}
 V_{0h} = \{ v\in C(\overline{Q})  : v(x_K(\cdot)) \in\mathbb{P}_p(\hat{K}),\,
\forall  K \in \mathcal{T}_h,\, v=0\;\mbox{on}\; {\overline \Sigma}
\}
\end{equation}
as usual \cite{LS:Braess:2013a,LS:BrennerScott:2008a,LS:Ciarlet:1978a},
where $x_K(\cdot)$ denotes the map from the reference element $\hat{K}$ to the finite element 
$K \in \mathcal{T}_h$, 
and $\mathbb{P}_p(\hat{K})$ is the space of polynomials of the degree $p$ on the reference element $\hat{K}$. The space-time finite elements space $V_{0h}$ is obviously a finite dimensional subspace 
of the space $H^1_0(Q) \subset \mathring{V}_2^{1,0}(Q)$.
Let us denote the dimension of $V_{0h}$ by $N_h$.
For simplicity, throughout this paper and, in particular, in our numerical experiments 
in Section~\ref{LS:Section:NumericalResults}, 
we use affine-linear mappings $x_K(\cdot)$ and 
simplicial elements $K$.

Following \cite{LS:LangerNeumuellerSchafelner:2019a}, we first multiply \eqref{LS:eq:div}, 
that is valid in $L_2(K)$ for all  $K \in \mathcal{T}_h$ since the triangulation $\mathcal{T}_h$ 
is align with $Q_i$, by a locally scaled 
upwind test function
\begin{equation*}
v_{h,K}(x,t) := v_h(x,t) + \theta_K h_K \partial_{t} v_h(x,t),\; v_h \in  V_{0h},
\end{equation*}
then integrate over $K$, and finally sum over all $K \in \mathcal{T}_h$, giving the identity
\begin{align*}
    \sum_{K \in \mathcal{T}_h} \int_{K}\!\mathrm{div}\begin{pmatrix}-\nu\nabla_{x}u\\u\end{pmatrix} 
    (v_h + \theta_K h_K \partial_{t} v_h) \;\mathrm{d}Q\\
    = 
    \sum_{K \in \mathcal{T}_h} \int_{K}\! (f-\mathrm{div}_x \mathbf{f})(v_h + \theta_K h_K \partial_{t} v_h)\;\mathrm{d}Q 
\end{align*}
Integration by parts yields
\begin{align*}
    \sum_{K \in \mathcal{T}_h} \int_{K}\!
    \left[
    \nu \nabla_x u \cdot \nabla_{x}v_h - u \partial_t v_h 
    + \theta_K h_K \mathrm{div}\begin{pmatrix}-\nu\nabla_{x}u\\u\end{pmatrix} 
    \partial_{t} v_h \;\mathrm{d}Q 
    \right]\\
    + \sum_{K \in \mathcal{T}_h} \int_{\partial K}\!
        \begin{pmatrix}-\nu\nabla_{x}u\\u\end{pmatrix} \cdot  
        \begin{pmatrix}n_{x}\\n_t\end{pmatrix} v_h\;\mathrm{d}s\\
    = 
    \sum_{K \in \mathcal{T}_h} \int_{K}\! 
    \left[
    f v_h + \mathbf{f} \cdot \nabla_x v_h + \theta_K h_K f \partial_{t} v_h 
            - \theta_K h_K\mathrm{div}_x \mathbf{f} \partial_{t} v_h 
    \right]  \;\mathrm{d}Q\\
    - \sum_{K \in \mathcal{T}_h} \int_{\partial K}\! \mathbf{f} \cdot n_x v_h \;\mathrm{d}s
\end{align*}
Since the triangulation $\mathcal{T}_h$ is aligned with the interfaces generated by the 
decomposition of $Q$ into the space-time subdomains $Q_1$, $Q_2$, \ldots, $Q_m$,
the space-time flux identity (\ref{LS:eq:SpaceTimeFluxH1}) is also valid if one replaces
the subdomains $Q_i$ by the finite elements $K$. Using (\ref{LS:eq:SpaceTimeFluxH1}) 
with $Q_i$ replaced by $K$ in the last identity, we immediately obtain 
the consistency relation
\begin{equation}\label{LS:eq:ConsistencyIdentity}
	a_h(u,v_h) = l_h(v_h) \quad\forall v_h\in V_{0h}, 
\end{equation}
with
\begin{align}
\label{LS:eq:bilinearform}
    a_h(u,v_h) :=& 
    \sum_{K\in\mathcal{T}_h} \int_{K}\left[\nu\nabla_{x}u \cdot \nabla_{x}v_h - u \partial_t v_h 
        + \theta_{K}h_K \mbox{div} \begin{pmatrix}-\nu\nabla_{x}u\\u\end{pmatrix}\partial_t v_h\right] \;\mathrm{d}Q\\ \nonumber
        & + \int_{\Omega} u(x,T) v_h(x,T) \mathrm{d}x,\\ \nonumber
     l_h(v_h) :=& \sum_{K\in\mathcal{T}_h}\int_{K}\left[fv_h + \mathbf{f}\cdot\nabla_{x}v_h  
                        + \theta_{K}h_K f \partial_t v_h
                        - \theta_{K}h_K \mbox{div}_x( \mathbf{f}) \, \partial_t v_h \right] \;\mathrm{d}Q\\  \nonumber
                        & + \int_{\Omega} u_0(x) v_h(x,0) \mathrm{d}x.
\end{align}
Therefore, we arrive at the following consistent space-time finite element scheme
for solving the initial-boundary value problem (\ref{LS:eq:SpaceTimeVF}):
find $ u_h\in V_{0h} $ such that 
\begin{equation}
\label{LS:eq:SpaceTimeFiniteElementScheme}
	a_h(u_h,v_h) = l_h(v_h) \quad\forall v_h\in V_{0h}. 
\end{equation}
Subtracting the space-time finite element scheme (\ref{LS:eq:SpaceTimeFiniteElementScheme})
form the consistency relation \eqref{LS:eq:ConsistencyIdentity}, we get the Galerkin orthogonality
relation
\begin{equation}
\label{LS:eq:GalerkinOrthogonality}
	a_h(u-u_h,v_h) = 0 \quad\forall v_h\in V_{0h},
\end{equation}
that is crucial for deriving a priori discretization error estimates in the next section.
We now show that the bilinear form $a_h(\cdot,\cdot)$ is coercive on $V_{0h}$ 
with respect to the norm
\begin{align}
\label{LS:eq:Norm_h}
\nonumber
  \|v_h\|_{h}^2 := \sum_{K\in\mathcal{T}_h} \bigl[ \|\nu^{1/2}\nabla_{x}v_h\|_{L_2(K)}^2 + 
                \theta_Kh_K\|\partial_tv_h\|_{L_2(K)}^2 \bigl]\\ 
                + \|v_h(\cdot,T)\|_{L_2(\Omega)}^2 + \|v_h(\cdot,0)\|_{L_2(\Omega)}^2.
\end{align}
\begin{lemma}
\label{LS:lemma:Coercivity}
If $\theta_K$ is chosen sufficiently small, then there exist a positive 
constant $\mu_c$ such that
\begin{equation}
\label{LS:eq:Coercivity}
 a_h(v_h,v_h) \ge \mu_c \, \|v_h\|_h^2, \; \forall v_h \in V_{0h}.
\end{equation}
More precisely, if $\theta_K \le h_K / c_{I,K,\nu}^2$ for all $K \in \mathcal{T}_h$, then $\mu_c$ can be chosen $1/2$,
where $c_{I,K}(\nu)$ denotes the constant from the inverse inequality
\begin{equation}
\label{LS:eq:InverseInequality}
    \|\mbox{div}_x(\nu \nabla_x v_h)\|_{L_2(K)} \le c_{I,K,\nu}\, h_k^{-1} \|\nu^{1/2} \nabla_x v_h\|_{L_2(K)}
\end{equation}
that holds for all $v_h \in V_{0h}$ and $K \in \mathcal{T}_h$; see \cite{LS:LangerNeumuellerSchafelner:2019a}.
\end{lemma}
\begin{proof}
For all $v_h \in V_{0h}$, we have the following representation of the bilinear form $a_h(v_h,v_h)$:
\begin{align}
\label{LS:eq:ah(vh,vh)}
\nonumber
 a_h(v_h,v_h) =& \sum_{K\in\mathcal{T}_h} \int_{K}\left[|\nu^{1/2}\nabla_{x}v_h|^2 - \frac{1}{2}\partial_t v_h^2 
        + \theta_{K}h_K \mbox{div} \begin{pmatrix}-\nu\nabla_{x}v_h\\v_h\end{pmatrix}\partial_t v_h\right] 
        \, \mathrm{d}Q\notag\\
               & +\int_{\Omega} |v_h(x,T)|^2\, \mathrm{d}x\notag\\
            =  & \sum_{K\in\mathcal{T}_h} \int_{K} \left[|\nu^{1/2}\nabla_{x}v_h|^2 
                                                            + \theta_{K}h_K |\partial_t v_h|^2 \right] \,\mathrm{d}Q 
                                                            + \frac{1}{2}\int_{\Omega} |v_h(x,T)|^2\, \mathrm{d}x\notag\\
               & + \frac{1}{2}\int_{\Omega} |v_h(x,0)|^2\, \mathrm{d}x 
                 - \sum_{K\in\mathcal{T}_h} \theta_{K}h_K  \int_{K}\mbox{div}_x(\nu\nabla_{x}v_h)\partial_t v_h \, \mathrm{d}Q
\end{align}
Estimating the last term by Cauchy's inequality, Young's inequality, and the inverse inequality (\ref{LS:eq:InverseInequality}),
we immediately get 
\begin{align*}
 a_h(v_h,v_h) \ge& \sum_{K\in\mathcal{T}_h} 
                    \left[(1 - \frac{\varepsilon \theta_{K} c_{I,K,\nu}^2}{2 h_K}) \|\nu^{1/2}\nabla_{x}v_h\|^2_{L_2(K)}
                                            + (1 - \frac{1}{2\varepsilon})\theta_{K} h_K \|\partial_t v_h\|^2_{L_2(K)} \right]\\
               & + \frac{1}{2}\|v_h(\cdot,T)\|^2_{L_2(\Omega)} + \frac{1}{2}\|v_h(\cdot,0)\|^2_{L_2(\Omega)} 
               , \; \forall v_h \in V_{0h},
\end{align*}
from which (\ref{LS:lemma:Coercivity}) follows for $\varepsilon =1$ and $\theta_K \le h_k / c_{I,K,\nu}^2$ for all $K \in \mathcal{T}_h$.
\end{proof}
\begin{remark}
For the special case that $p=1$ (linear shape functions) and $\nu = \nu_K = \mbox{const} $ for all $K \in \mathcal{T}_h$,
the last term in (\ref{LS:eq:ah(vh,vh)}) vanishes since 
$\mbox{div}_x(\nu\nabla_{x}v_h) = \nu_K \Delta_x v_h = 0$ on $K$. 
Therefore, we have 
\begin{align}
\label{LS:eq:ah(vh,vh):p=1}
\nonumber
 a_h(v_h,v_h) 
            =  & \sum_{K\in\mathcal{T}_h} \int_{K} \left[|\nu^{1/2}\nabla_{x}v_h|^2 
                                                            + \theta_{K}h_K |\partial_t v_h|^2 \right] \,\mathrm{d}Q 
                                                            + \frac{1}{2}\int_{\Omega} |v_h(x,T)|^2\, \mathrm{d}x\notag\\
               & + \frac{1}{2}\int_{\Omega} |v_h(x,0)|^2\, \mathrm{d}x 
                 =: \||v_h\||^2_h , \; \forall v_h \in V_{0h},
\end{align}
i.e. the bilinear form $a_h(\cdot,\cdot)$ is  coercive on $V_{0h}$ wrt the norm $\||\cdot\||_h$ 
with the constant $\mu_c =1$.
\end{remark}

The coercivity of the bilinear form $a_h(\cdot,\cdot)$ on $V_{0h}$ yields uniqueness,
and, in the finite-dimensional case, uniqueness implies existence.
Thus, the space-time finite element scheme (\ref{LS:eq:SpaceTimeFiniteElementScheme})
has a unique solution $u_h \in V_{0h}$ that can be found by solving one linear system 
of algebraic equations. 
Once the basis is chosen, the space-time finite element scheme (\ref{LS:eq:SpaceTimeFiniteElementScheme})
is nothing but one system of linear algebraic equations. 
Indeed, let 
$ \{p^{(j)}:j=1,\ldots,N_h\}$
be the finite element nodal basis of $ V_{0h}$, 
i.e., $V_{0h} = \mbox{span}\{p^{(1)},\ldots,p^{(N_h)} \}$,
where $N_h$ is the number of all space-time unknowns (dofs).
Then we can express the approximate solution $ u_h $ in terms of this basis, 
i.e., $ u_{h}(x,t) = \sum_{j=1}^{N_h} u_j\,p^{(j)}(x,t) $.
Inserting this ansatz into (\ref{LS:eq:SpaceTimeFiniteElementScheme}) and testing with $p^{(i)}$,
we get the linear system
\begin{equation}
\label{LS:eq:LinearSystem}
{K}_{h} \underline{u}_{h}= \underline{f}_{h},
\end{equation}
for determining the unknown coefficient vector $\underline{u}_{h}= (u_j)_{j=1,\ldots,N_h} \in \mathbb{R}^{N_h} $,
where $ {K}_{h} = (a_h(p^{(j)},p^{(i)}))_{i,j=1,\ldots,N_h} $ and 
$ \underline{f}_{h} = (l_h(p^{(i)}))_{i=1,\ldots,N_h}$.
The system matrix $ {K}_{h}$ is non-symmetric, but positive definite due to Lemma~\ref{LS:lemma:Coercivity}.

\section{A priori error estimates}
\label{LS:Section:Apriori}

We now introduce the norm 
\begin{align}
\nonumber
    \|u\|_{h,*}^2 :=& \sum_{K\in\mathcal{T}_h} \left[ \|\nu^{1/2}\nabla_{x}u\|_{L_2(K)}^2 
    + (\theta_Kh_K)^{-1}\|u\|_{L_2(K)}^2\right.\\ \nonumber
    & +\left.  \theta_Kh_K\|\mathrm{div} \begin{pmatrix}-\nu\nabla_{x}u\\u\end{pmatrix}\|_{L_2(K)}^2 \right] 
    + \|u(\cdot,T)\|^2_{L_2(\Omega)}
\end{align}
that is defined on $ V_{0h,*} = H^{\mathcal{L}}_{0}(\mathcal{Q}) + V_{0h}$ to which the solution $u$ and the 
discretization error $u - u_h$ belong. The solution space $H^{\mathcal{L}}_{0}(\mathcal{Q})$ can be chosen as
$\{v \in H^{1,0}_0(Q):\, 
\mathcal{L} u := \mathrm{div} (-\nu\nabla_{x}u,u) 
\in L_2(Q_i),\,
i=1,\ldots,m\}$.
\begin{lemma}
\label{LS:lemma:Boundedness}
    The bilinear form \eqref{LS:eq:bilinearform} is bounded on $ V_{0h,*}\times V_{0h} $, i.e.,
    \begin{equation}\label{LS:eq:GeneralizedBoundedness}
        |a_h(u,v_h)| \le \mu_b \, \|u\|_{h,*} \|v_h\|_h, \; \forall u \in V_{0h,*}, v_h \in V_{0h}
    \end{equation}
with $\mu_b = 1$.
\end{lemma}
\begin{proof}
Using Cauchy's inequalities, we immeadiately get   
    \begin{align*}
        a_h(u,v_h) =& \sum_{K\in\mathcal{T}_h} \int_{K}\!\nu\nabla_{x}u\cdot\nabla_{x}v_h - u\partial_{t}v_h + \theta_Kh_K \mathrm{div}\begin{pmatrix}-\nu\nabla_{x}u\\u\end{pmatrix}\partial_{t}v_h \mathrm{d}Q \\
        &\qquad+ \int_{\Omega}u(x,T)v_h(x,T)\mathrm{d}x \\
        \leq& \sum_{K\in\mathcal{T}_h} \bigg[\left(\|\nu^{1/2}\nabla_{x}u\|_{L_2(K)}^2\right)^{1/2}\left(\|\nu^{1/2}\nabla_{x}v_h\|_{L_2(K)}^2\right)^{1/2} \\
        &\qquad+ \left((\theta_Kh_K)^{-1}\|u\|_{L_2(K)}^2\right)^{1/2}\left(\theta_Kh_K \|\partial_{t}v_h\|_{L_2(K)}^2\right)^{1/2} \\
        &\qquad+ \left(\theta_Kh_K \|\mathrm{div}\begin{pmatrix}-\nu\nabla_{x}u\\u\end{pmatrix}\|_{L_2(K)}^2\right)^{1/2} \left(\theta_Kh_K \|\partial_{t}v_h\|_{L_2(K)}^2 \right)^{1/2} \bigg] \\
        &\qquad+ \left(\|u(\cdot,T)\|_{L_2(\Omega)}^2\right)^{1/2} \left(\|v_h(\cdot,T)\|_{L_2(\Omega)}^2\right)^{1/2}\\
        \leq& \mu_b \|u\|_{h,*} \|v_h\|_{h}
    \end{align*}
with $\mu_b = 1$.
\end{proof}
The Galerkin orthogonality \eqref{LS:eq:GalerkinOrthogonality} together with the coercivity \eqref{LS:eq:Coercivity} 
and the generalized boundedness \eqref{LS:eq:GeneralizedBoundedness} of the bilinear form $a_h(\cdot, \cdot)$ 
immediately yield a best-approximation 
(C\'{e}a-like) 
discretization error estimate in the 
corresponding norms.
\begin{lemma}
\label{LS:lemma:Cea}
Let 
$u \in H^{\mathcal{L}}_{0}(\mathcal{Q})$
and $u_h \in V_{0h}$ be the solutions of the parabolic IBVP \eqref{LS:eq:SpaceTimeVF} 
and the space-time finite element scheme \eqref{LS:eq:SpaceTimeFiniteElementScheme}, respectively,
and let the assumptions of Lemma~\ref{LS:lemma:Coercivity} (coercivity) Lemma~\ref{LS:lemma:Boundedness} (boundedness) hold.
Then the discretization error estimate 
\begin{equation}
\label{LS:eq:Cea}
\|u-u_h\|_h \leq \inf_{v_h\in V_{0h}} \left(\|u-v_h\|_h + \frac{\mu_b}{\mu_c}\|u-v_h\|_{h,*}\right)
\end{equation}
holds.
\end{lemma}
\begin{proof}
Using \eqref{LS:eq:Coercivity}, \eqref{LS:eq:GalerkinOrthogonality}, and
\eqref{LS:eq:GeneralizedBoundedness}, we immediately get the estimates
\begin{equation*}
 \mu_c \|v_h-u_h\|_h^2 \le a_h(v_h-u_h,v_h-u_h) = a_h(u-u_h,v_h-u_h) \le \mu_b \|u-u_h\|_h\|u-v_h\|_{h,*}
\end{equation*}
for any $v_h\in V_{0h}$, from which \eqref{LS:eq:Cea} follows by triangle inequality.
\end{proof}
In order to obtain a convergence rate estimate from (\ref{LS:eq:Cea}), we have to insert 
a suitable interpolation or quasi-interpolation 
$I_{h}^p(u) \in V_{0h}$ 
of the solution $u$ 
into the infimum. If $u$ is sufficiently smooth, then we can use the nodal Lagrange 
interpolation that leads to completely local estimates. 
More precisely, let $u \in H^{\mathcal{L}}_{0}(\mathcal{Q}) \cap H^{k}(Q) \cap  H^{l}(\mathcal{T}_h)$  
with some $l \ge k > (d+1)/2$. Since in this case $H^{k}(Q)\subset C(\overline{Q})$,
the nodal Lagrange interpolation is well defined.
Then the standard finite element interpolation error estimates
(see, e.g., \cite[Theorem 4.4.4]{LS:BrennerScott:2008a} or \cite[Theorem 3.1.6]{LS:Ciarlet:1978a}) and the best-approximation error estimate (\ref{LS:eq:Cea}) 
lead to the a priori discretization error estimate
\begin{equation}
	\| u-u_h \|_h \leq c\left(\sum_{K\in\mathcal{T}_h}h_K^{2(s-1)}| u |_{H^{s}(K)}^2\right)^{1/2} \label{LS:eq:apriorierrorestimatesmooth}
\end{equation}
with $s = \min \{l,p+1\}$ and some generic positive constant $c$;
see \cite[Theorem 13.3]{LS:LangerNeumuellerSchafelner:2019a}.
If the nodal Lagrange interpolation is not well defined due to the low regularity of 
the solution $u$, we can make use of the quasi-interpolators proposed 
by Cl\'{e}ment \cite{LS:Clement:1975a} or Scott and Zhang \cite{LS:ScottZhang:1990a}.
Let $ v\in H^k(Q) $, with some 
$k > 1$.
Then, for instance, the Scott-Zhang quasi-interpolation operator $ I_h^{S\!Z} : H^1(Q) \rightarrow V_{0h} $ 
provides the local interpolation error estimate 
\begin{equation}\label{LS:eq:LocalInterpolationErrorEstimateSZ}
    \| v-{I}_h^{S\!Z} v \|_{H^m(K)} \leq c h_K^{s-m} |v|_{H^s(S_K)},\;\; m=0,1,
\end{equation}
where $s=\min\{k,p+1\}$, $c$ is again a generic positive constant, and 
$ S_K := \{ K'\in\mathcal{T}_h : \overline{K}\cap\overline{K}'\neq\emptyset\}$ 
denotes the neighborhood of the simplex $K \in \mathcal{T}_h$;
see 
monograph 
\cite[(4.8.10)]{LS:BrennerScott:2008a})
or the original paper \cite{LS:ScottZhang:1990a},
and space interpolation results \cite{LS:BerghLoefstroem:1976a,LS:Heuer:2014a}.
\begin{theorem}
\label{LS:theorem.ConvergenceRateEstimate}
Beside the assumptions imposed on the data $\nu$, $f$, $\mathbf{f}$ and $Q$, and the triangulation $\mathcal{T}_h$,
we assume that the solution $u$ of the space-time variational 
problem \eqref{LS:eq:SpaceTimeVF} belongs to  $H^{\mathcal{L}}_{0}(\mathcal{Q}) \cap H^k(Q)$, $(d+1)/2 \ge k > 1$, and $\nu \nabla_x(u) \in (H^{k-1}(K))^d$ 
for all $K \in \mathcal{T}_h$.
Furthermore, let the assumptions of Lemma~\ref{LS:lemma:Cea} be fulfilled.
Then the a priori discretization error estimate
\begin{equation}
\label{LS:eq:discretization-error-estimate}
    \|u-u_h\|_h^2 \le c\,
    \sum_{K\in\mathcal{T}_h} \left[h_K^{2(s-1)} \bigl(|u|_{H^s(S_K)}^2  + |\nu \nabla_x u|_{H^{s-1}(K)}^2\bigr)   
            + h_K^2\|\mathrm{div}_x(\nu\nabla_{x}u)\|_{L_2(K)}^2\right]
\end{equation}
    holds,  with $ s=\min\{k,p+1\} = k \le (d+1)/2$ for $d=1,2,3$ and a positive generic constant $ c $.
\end{theorem}
\begin{proof}
Inserting $v_h = I_h^{S\!Z}u \in V_{0h}$ into the infimum of the right-hand side of the best-approximation
error estimate \eqref{LS:eq:Cea}, we immediately get
\begin{align}
\label{LS:eq:dee1}
\nonumber
\|u-u_h\|_h^2   \leq& \left(\|u-I_h^{S\!Z}u\|_h + \frac{\mu_b^2}{\mu_c^2}\|u-I_h^{S\!Z}u\|_{h,*}\right)^2\\ \nonumber
                \leq& \; 2 \|e_h\|_h^2 + 2 \frac{\mu_b^2}{\mu_c^2}\|e_h\|_{h,*}^2\\ \nonumber
                \leq&   \sum_{K\in\mathcal{T}_h} \bigl[ 
                        2(1+\frac{\mu_b^2}{\mu_c^2})\| \nu^{1/2}\nabla_{x}e_h\|_{L_2(K)}^2 + 
                        2 \theta_Kh_K\|\partial_t e_h\|_{L_2(K)}^2\\\nonumber
                    &   + 2 \frac{\mu_b^2}{\mu_c^2}(\theta_Kh_K)^{-1}\|e_h\|_{L_2(K)}^2
                        + 2 \frac{\mu_b^2}{\mu_c^2}\theta_Kh_K\|\mathrm{div} \begin{pmatrix}-\nu\nabla_{x}e_h\\e_h\end{pmatrix}\|_{L_2(K)}^2
                        \bigl]\\ 
                    & + 2(1+\frac{\mu_b^2}{\mu_c^2})\|e_h(\cdot,T)\|_{L_2(\Omega)}^2 + 2 \|e_h(\cdot,0)\|_{L_2(\Omega)}^2,
\end{align}
where $e_h = u - I_h^{S\!Z}u$ denotes the interpolation error.
The first three terms in the sum can directly be estimated by means of the local interpolation error estimate 
\eqref{LS:eq:LocalInterpolationErrorEstimateSZ}. 
The fourth term needs special treatment. Inserting and subtracting the average
\begin{equation}
    I_{h}^0 q(x) = \frac{1}{|K|} \int_K q(y) dy \in \mathbb{R}^{d+1} \; \mbox{for}\, x \in K 
\end{equation}
of the space-time flux $q = (-\nu (\nabla_x u)^T,u)^T$ over $K \in\mathcal{T}_h$
and using the inverse inequalities 
$\| \mathrm{div}_x(c_h + \nu \nabla_x v_h)\|_{L_2(K)} \le c h_K^{-1} \|c_h + \nu \nabla_x v_h\|_{L_2(K)}$
for all $(c_h,v_h) \in \mathbb{R}^d \times V_{0h}$ and
$\| \partial_t (c_h - v_h)\|_{L_2(K)} \le c h_K^{-1} \|c_h -  v_h\|_{L_2(K)}$
for all $(c_h,v_h) \in \mathbb{R} \times V_{0h}$, we can estimate this term as follows: 
\begin{align*}
 \|\mathrm{div} \begin{pmatrix}-\nu\nabla_{x}e_h\\e_h\end{pmatrix}\|_{L_2(K)}
    &=\|\mathrm{div} (q - I_{h}^0 q + I_{h}^0 q - \begin{pmatrix}-\nu\nabla_{x}(I_h^{S\!Z}u)\\I_h^{S\!Z}u\end{pmatrix})\|_{L_2(K)}\\
  \le&\|\mathrm{div} q\|_{L_2(K)} + \|\mathrm{div} (I_{h}^0 q -\begin{pmatrix}-\nu\nabla_{x}(I_h^{S\!Z}u)\\I_h^{S\!Z}u\end{pmatrix})\|_{L_2(K)}\\
  \le&\|\mathrm{div} q\|_{L_2(K)} + c h_K^{-1} \| I_{h}^0 q -\begin{pmatrix}-\nu\nabla_{x}(I_h^{S\!Z}u)\\I_h^{S\!Z}u\end{pmatrix}\|_{L_2(K)}.
\end{align*}
Now, the last term can be estimated by means of the local approximation properties of $I_{h}^0$ and $I_h^{S\!Z}$.
Indeed, inserting and subtracting $q$ in the last term of the above estimate, we get
\begin{align*}
    \| I_{h}^0 q &-\begin{pmatrix}-\nu\nabla_{x}(I_h^{S\!Z}u)\\I_h^{S\!Z}u\end{pmatrix}\!\|_{L_2(K)}^2
        \le 
        2 \| I_{h}^0 q - q\|_{L_2(K)}^2 + 
    2 \| q-\begin{pmatrix}-\nu\nabla_{x}(I_h^{S\!Z}u)\\I_h^{S\!Z}u\end{pmatrix}\!\|_{L_2(K)}^2\\
    \le & \; c 
    \left( 
        h_K^{2(s-1)} |\nu \nabla_{x}(u) |_{H^{s-1}(K)}^2 + 
        h_K^2 |u|_{H^1(K)}^2 +
        h_K^{2(s-1)} |u|_{H^{s}(S_K)}^2 +
        h_K^{2s} |u|_{H^{s}(S_K)}^2 
        \right),
\end{align*}
    where we have used the Poincar\'{e} inequality for estimating the first term
    and \eqref{LS:eq:LocalInterpolationErrorEstimateSZ} for the second term.
    Finally, we have to estimate the last two terms in \eqref{LS:eq:dee1}. 
    The trace theorem and the local interpolation error estimate 
    \eqref{LS:eq:LocalInterpolationErrorEstimateSZ} yield the estimates
    \begin{align*}
            \|e_h(\cdot,T)\|_{L_2(\Omega)}^2
            \le &\, c\, \|e_h\|_{H^1(Q)}^2     \le  c \sum_{K\in\mathcal{T}_h} \|e_h\|_{H^1(K)}^2
            \le  c \sum_{K\in\mathcal{T}_h} h_K^{2(s-1)} |u|_{H^s(S_K)}^2
    \end{align*}
that is obviously also valid for the last term $\|e_h(\cdot,0)\|_{L_2(\Omega)}^2$ in \eqref{LS:eq:dee1}. 
    Combining these estimates, we arrive at \eqref{LS:eq:discretization-error-estimate}.
\end{proof}
%
\section{A posteriori error estimates and adaptivity}
\label{LS:Section:Aposteriori}

In contrast to elliptic boundary value problems, there are not so many results on space-time
a posteriori error estimates and simultaneous space-time adaptivity driven by the corresponding error indicators 
for finite element schemes on unstructured simplicial meshes that treat  time $t$ as just another 
variable, say, $x_{d+1}$, like in this paper; cf. also survey paper \cite{LS:SteinbachYang:2019a}. 
As in the elliptic case, Steinbach and Yang \cite{LS:SteinbachYang:2018a} 
have recently proposed to use the local residual error indicator
\begin{equation}
	\label{LS:eq:ResidualErrorIndicator}
	\eta_K := \left( h_K^2 \|R_h(u_h)\|_{L_2(K)}^2 + h_K \|J_h(u_h) \|_{L_2(\partial K)}^2 \right)^{1/2}
\end{equation}
for driving the adaptivity, where $ u_h $ is the space-time finite element solution,
$R_h(u_h) := f + \mathrm{div}_x(\nu \nabla_x u_h) - \partial_t u_h$ denotes the 
residual on  $K \in \mathcal{T}_h $, 
$J_h(u_h) := [\nu \nabla_x u_h]_e$ represents the jump of the space flux across one face $e \subset \partial K$ 
of the boundary $\partial K$ of an element $K \in \mathcal{T}_h $.
Here, we assume that there are no distributional sources, i.e., $\mathbf{f} = \mathbf{0}$.
While
one has complete theoretical control  
on the adaptive process (reliability, efficiency, convergence, optimality) 
in the 
elliptic case, see \cite{LS:CarstensenFeischlPagePraetorius:2014a} and the references therein, 
similar theoretical results are not available 
for adaptive processes based on the residual error indicator \eqref{LS:eq:ResidualErrorIndicator}.
Even the reliability is not shown theoretically, although the numerical results presented in 
\cite{LS:SteinbachYang:2018a,LS:SteinbachYang:2019a} show the same behavior as in the elliptic case.

Repin proposed functional a posteriori discretization error estimates for the instationary heat equation 
in \cite{LS:Repin:2002a}, see also Repin's monograph \cite{LS:Repin:2008a}[Section~9.3, pp. 229--242].
Repin's first form of the error majorant
\begin{align*}
    &\overline{\mathfrak{M}}_{1}^2(\beta,\delta,v,\mathbf{y}) := \int_\Omega |v(x,0)-u_0(x)|^2\, \mathrm{d}x\\  
        &+ \frac{1}{\delta} \int_{Q} \left[ (1+\beta) | \mathbf{f} + \mathbf{y} - \nu \nabla_{x} v|^2 + c_{F\Omega}^2 (1+\frac{1}{\beta}) |f-\partial_{t}v+\mathrm{div}_{x}\mathbf{y}|^2 \right]\mathrm{d}Q
\end{align*}
provides a guaranteed upper bound of the error $u - v$ with respect to the norm 
\begin{equation}
\label{LS:eq:FunctionalErrorEstimate1}
 |\!|\!| u - v |\!|\!|_{(1,2 - \delta)}^2 :=   (2-\delta)\|\nabla_{x}(u-v)\|_{L_2(Q)}^{2} + 
 \|u -v \|_{L_2(\Sigma_T)}^{2} 
 \leq \overline{\mathfrak{M}}_{1}^2(\beta,\delta,v,\mathbf{y})
\end{equation}
for any approximation $v \in H^1_0(Q)$ to the solution $u$ of \eqref{LS:eq:SpaceTimeVF},
for any flux $\mathbf{y} \in H(\mathrm{div}_x,Q) := \{ {\mathbf{y}} \in (L_2(Q))^d: \mathrm{div}_x(\mathbf{y}) \in L_2(Q)\}$,
and for any weight function 
$\beta \in L^\infty_\mu(0,T) := \{\beta \in L^\infty(0,T):\,\beta(t) \ge \mu \mbox{ for almost all } t \in (0,T)\}$.
Here $\delta \in (0,2]$ and $\mu \in (0,1)$ are fixed parameters, and $c_{F\Omega}$ is the Friedrichs constant
for the space domain $\Omega$. 
Estimate \eqref{LS:eq:FunctionalErrorEstimate1} is proven
in Theorem~9.6 from \cite{LS:Repin:2008a}
for the case $\mathbf{f}=\mathbf{0}$, but the proof remains unchanged in the presence 
of distributional sources.
In particular, we can choose our finite element solution $u_h \in V_{0h} \subset H^1_0(Q)$ as 
$v$ in \eqref{LS:eq:FunctionalErrorEstimate1}. The parameter $\delta \in (0,2]$ 
can be chosen to weight the two parts of the norm $|\!|\!| \, \cdot \, |\!|\!|_{(1,2 - \delta)}$,
e.g., the choice $\delta = 1$ equilibrates both parts, whereas the choice $\delta = 2$ provides 
a $L_2$ error estimate at the final time $t=T$.
In order to compute an upper bound on the discretization error $u - u_h$ 
via estimate \eqref{LS:eq:FunctionalErrorEstimate1}, we have 
to choose appropriate fluxes $\mathbf{y} \in H(\mathrm{div}_x,Q)$ and 
weight functions $\beta \in L^\infty_\mu(0,T)$.
We cannot choose $\mathbf{y} = \nu \nabla_x u_h - \mathbf{f}$
since the finite element flux $\nu \nabla_x u_h$ does not belong to $H(\mathrm{div}_x,Q)$ in general.
Thus, we have to postprocess the finite element flux in such a way that the postprocessed 
finite element flux $P_h (\nu \nabla_x u_h)$ belongs to some $Y_h \subset H(\mathrm{div}_x,Q)$.
In our numerical experiments in Section~\ref{LS:Section:NumericalResults}, we simply average 
the finite element flux in the vertices; see e.g. \cite{LS:ZienkiewiczZhu:1992a,LS:ZienkiewiczZhu:1992b}.
Then one can calculate $\beta$ from the minimization of the majorant.
If we choose $\beta$ as a positive constant, then we get 
\begin{equation}
\label{LS:eq:beta}
\beta^{(0)} = c_{F\Omega} \, \overline{\mathfrak{M}}_{eq}(u_h,\mathbf{y}_h^{(0)})\, /\, \overline{\mathfrak{M}}_{flux}(u_h,\mathbf{y}_h^{(0)})   
\end{equation}
where $\mathbf{y}_h^{(0)} = P_h (\nu \nabla_x u_h) - \mathbf{f}$, 
$\overline{\mathfrak{M}}_{eq}(u_h,\mathbf{y}_h^{(0)}) = \int_{Q}  |f-\partial_{t}u_h+\mathrm{div}_{x}\mathbf{y}_h^{(0)}|^2 \mathrm{d}Q$,
and 
$\overline{\mathfrak{M}}_{flux}(u_h,\mathbf{y}_h^{(0)}) = \int_{Q} |\mathbf{f} + \mathbf{y}_h^{(0)} - \nu \nabla_x u_h|^2 \mathrm{d}Q$.
If one is not satisfied with the bound $\overline{\mathfrak{M}}_{1}^2(\beta^{(0)},\delta,u_h,\mathbf{y}_h^{(0)})$,
then one can start an alternating minimization of the majorant 
$\overline{\mathfrak{M}}_{1}^2(\beta,\delta,u_h,\mathbf{y}_h)$ with respect to $\mathbf{y}_h$ and $\beta$,
where one has to solve  the auxiliary problem: find $\mathbf{y}_h^{(k+1)}\in Y_h$ such that
\begin{equation}\label{LS:eq:FunctionalEstimatorAuxillaryProblem}
\mathbf{A}_k(\mathbf{y}_h^{(k+1)},\mathbf{w}_h) = \mathbf{F}_k(\mathbf{w}_h)\; \forall \mathbf{w}_h \in Y_h
\end{equation}
for improving the fluxes followed by calculating improved weights $\beta^{(k+1)}$ 
via formula \eqref{LS:eq:beta}, 
where $Y_h \subset H(\mathrm{div}_x,Q)$ is a suitable finite element space for the fluxes, 
\begin{align*}
    \mathbf{A}_k(\mathbf{y}_h,\mathbf{w}_h) := & \sum_{K\in\mathcal{T}_h} \int_{K}\!\Bigl( (1+\beta^{(k)}) \mathbf{y}_h\cdot\mathbf{w}_h \\
    &\qquad\qquad+ c_{F\Omega}^{2} (1+1/\beta^{(k)}) \mathrm{div}_x(\mathbf{y}_h)\mathrm{div}_x(\mathbf{w}_h)\Bigr)\;\mathrm{d}Q,\\
    \mathbf{F}_k(\mathbf{w}_h) := & \sum_{K\in\mathcal{T}_h} \int_{K}\!\Bigl((1+\beta^{(k)}) (\nu\nabla_{x}u_h-\mathbf{f})\cdot \mathbf{w}_h \\
    &\qquad\qquad- c_{F\Omega}^{2} (1+1/\beta^{(k)}) (f-\partial_{t}u_h)\,\mathrm{div}_x(\mathbf{w}_h)\Bigr)\;\mathrm{d}Q,
\end{align*}

In the numerical experiments presented in Section~\ref{LS:Section:NumericalResults}, we choose $ \delta = 1$, and, as already mentioned above, $ v = u_h $. 
Furthermore, we
use the 
nodal
averaging procedure to post-process the finite element flux $\nu\nabla_{x}u_h$. We use the finite element space 
$Y_h = (V_h)^d \subset (H^{1}(Q))^{d} \subset H(\mathrm{div}_x,Q)$
for the fluxes.
Note that this choice might be too restrictive since it excludes fluxes that are only continuous in their normal component. 
We apply one iteration of the alternating minimization described above, i.e., we have to solve the auxiliary problem \eqref{LS:eq:FunctionalEstimatorAuxillaryProblem} once, which is symmetric and positive definite. Hence, we can apply a few iterations of the (preconditioned)  Conjugate Gradient (CG) method, where the post-processed flux 
$\mathbf{y}^{(0)}_h$ is used as initial guess. We then use the first part of the majorant 
\[
    \eta_K(u_h) = \|\mathbf{f} + \mathbf{y}^{(1)}_h - \nu\nabla_{x}u_h\|_{L_2(K)},
\]
as an error indicator,
where $\mathbf{y}^{(1)}_h$ results from the improvement of $\mathbf{y}^{(0)}_h$ 
by a few CG iterations.
We proceed by using D\"{o}rfler marking \cite{LS:Doerfler:1996a} as a marking strategy, i.e., we determine a set $\mathcal{M}\subseteq\mathcal{T}_h$ of (almost) minimal cardinality such that 
\[ \Xi\ \eta(u_h)^2 \leq \sum_{K\in\mathcal{M}} \eta_K(u_h)^2, \]
with a suitable bulk parameter $\Xi \in (0,1]$. In particular, we use the implementation recently proposed by Pfeiler et al. \cite{LS:PfeilerPraetorius:2019a}, which is straightforward to parallelize. 
We will also survey the efficiency index of the a posteriori error estimator, that is defined as
\[
    \mathrm{I_{eff}}^2 = \frac{\sum_{K\in\mathcal{T}_h}\eta_K(u_h)^2}{\|u-u_h\|_h^2},
\]
with the exact solution $u$ and its finite element approximation $u_h$.

\section{Numerical results}
\label{LS:Section:NumericalResults}

In this section, we present the results of an extensive set of numerical experiments. 
We realized 
the space-time FEM in our \texttt{C++} code ``SpTmFEM'', using the 
fully parallelized 
finite element library MFEM \cite{LS:mfem-library}, 
where we implemented 
arbitrary-order finite elements in 4D. The large linear systems arising from the discretized variational problem are solved by means of a flexible Generalized Minimal Residual (GMRES) method \cite{LS:Saad:1993a}, preconditioned by an algebraic multigrid V-cycle. In particular, we use the 
\textit{BoomerAMG}, provided by the solver library hypre\footnote{https://computing.llnl.gov/projects/hypre-scalable-linear-solvers-multigrid-methods}. We start the iterative solver with the initial guess zero, and stop once the initial residual is reduced by a factor of \num{e-8}. However, especially in an adaptive procedure, it may be better to interpolate the solution from the previous refinement level to the current one and use it as an initial guess for the iterative solver. In addition, we relax the stopping criterion, i.e., we stop once the initial residual is reduced by a factor of e.g. \num{e-2}. This approach is sometimes referred to as \emph{Nested Iterations (NI)}. 

All numerical experiments were performed on the distributed memory cluster {\texttt{Quartz}\footnote{https://hpc.llnl.gov/hardware/platforms/Quartz}}, located at the Lawrence Livermore National Laboratory. 

%
%
\subsection{Moving peak}
\label{LS:Subsection:MovingPeak}
%
For our first example, we consider the four dimensional hyper-cube $Q=(0,1)^4$ as space-time cylinder, 
the
diffusion coefficient $\nu \equiv 1$, and we use the manu\-factured solution
\[
    u(x,t) = \prod_{i=1}^{3}(x_{i}^2-x_{i})(t^2-t)e^{-100\left((x_1-t)^2+(x_2-t)^2+(x_3-t)^2\right)}.
\]
We then
compute the right-hand side and boundary data accordingly. 
This solution is very smooth. 
Thus, we would expect optimal convergence rates for uniform mesh refinement. 
However, the solution is also very localized and has steep gradients. 
In order to recover the optimal convergence rate  quickly,
we will apply adaptive mesh refinement based on the error estimator from the previous section. 
Moreover, we compare the efficiency index $I_\mathrm{eff}$ of the functional error estimator with that of the residual-based error indicator proposed by Steinbach and Yang \cite{LS:SteinbachYang:2019a}.
In Figure~\ref{LS:fig:moving-peak:rates}, we present the relative error in the energy norm \eqref{LS:eq:Norm_h}, using both uniform and adaptive refinement, as well as different polynomial degrees for the finite element 
shape functions. 
We can observe the expected behavior, i.e., even a sufficiently uniformly refined mesh results in an (almost) optimal convergence rate wrt the polynomial degree. However, adaptive refinement results in optimal convergence rates, and, moreover, reduces the number of dofs needed to obtain an error below a certain threshold, e.g., for $p=3$, a relative error below 1\% needs only \num{4742845} adaptive 
space-time dofs, 
compared to 
more than \num{81044161} 
dofs using uniform refinement.
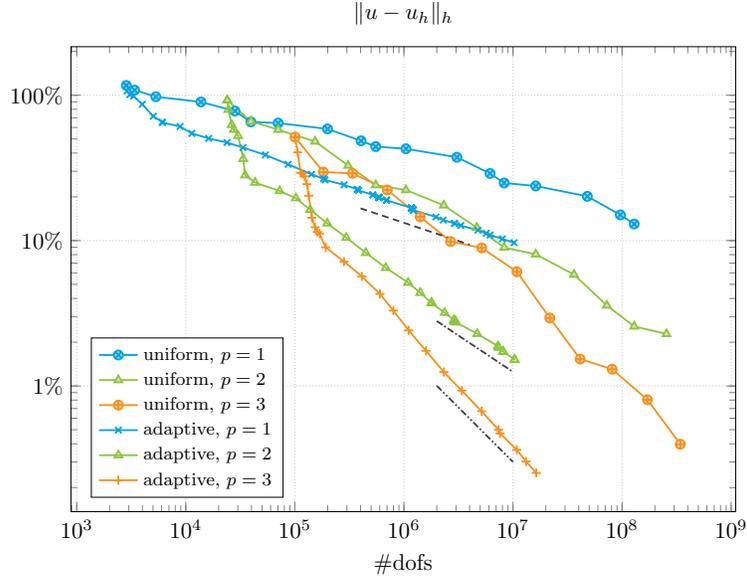
\begin{figure}[!htb]
    \centering%
    \begin{subfigure}{.85\linewidth}%
        \resizebox*{\linewidth}{!}{%
        \begin{tikzpicture}
            \begin{loglogaxis}[%
            width=\linewidth,scale only axis,%
            legend pos=south west,
            legend style={font=\footnotesize},%
            title=$ \|u-u_h\|_h $, xlabel={\#dofs},%
            cycle list name=adaptivity,%
            log y ticks with percent,%
            ]
            \addplot table [x=dofs, y=h-err, col sep=comma] {anc/moving-peak_4d_errors_o1_uniform.csv};
            \addplot table [x=dofs, y=h-err, col sep=comma] {anc/moving-peak_4d_errors_o2_uniform.csv};
            \addplot table [x=dofs, y=h-err, col sep=comma] {anc/moving-peak_4d_errors_o3_uniform.csv};
            \addplot table [x=dofs, y=h-err, col sep=comma] {anc/moving-peak_4d_errors_o1_ni_amr+functional+doerfler-0.25.csv};
            \addplot table [x=dofs, y=h-err, col sep=comma] {anc/moving-peak_4d_errors_o2_ni_amr+functional+doerfler-0.25.csv};
            \addplot table [x=dofs, y=h-err, col sep=comma] {anc/moving-peak_4d_errors_o3_ni_amr+functional+doerfler-0.25.csv};
            \addplot[darkgray,domain=4e5:4e6, densely dashed] {(x/4.7e6)^(-1./4)*0.9e-1};
            \addplot[darkgray,domain=2e6:1e7, densely dashdotted] {(x/1e7)^(-2./4)*1.25e-2};
            \addplot[darkgray,domain=2e6:1e7, densely dashdotdotted] {(x/1e7)^(-3./4)*3.e-3};
            \legend{{uniform, $ p=1 $},%
                    {uniform, $ p=2 $},%
                    {uniform, $ p=3 $},%
                    {adaptive, $ p=1 $},%
                    {adaptive, $ p=2 $},%
                    {adaptive, $ p=3 $},%
                }
            \end{loglogaxis}
        \end{tikzpicture}
        }
    \end{subfigure}%
    \caption{Example~\ref{LS:Subsection:MovingPeak}: Convergence rates of uniform and adaptive refinements, using the functional estimator, for $d=3$ with marking threshold $\Xi = 0.25$.}
    \label{LS:fig:moving-peak:rates}
\end{figure}
In the left of Figure~\ref{LS:fig:moving-peak:efficiencyAndPlot}, we plot the efficiency indices of the functional estimator and the residual indicator for different polynomial degrees. For the functional estimator, we observe efficiency indices $I_{\mathrm{eff}} \sim 1.4$, independently of the polynomial degree $p$.
The efficiency index of the residual indicator depends on $p$. For $p=1$, we observe efficiency indices around $1$ whereas the efficiency indices are much worse for higher polynomial
degrees $p$.
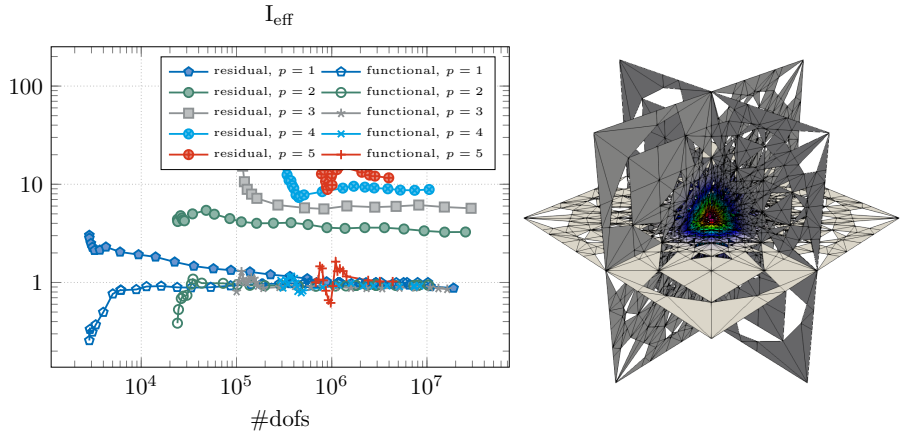
\begin{figure}[!htb]
    \centering%
    \begin{subfigure}{.575\linewidth}%
        \resizebox*{\linewidth}{!}{%
        \begin{tikzpicture}
            \begin{loglogaxis}[%
                width=\linewidth,scale only axis,%
                legend pos=north east,%
                legend columns=5,%
                transpose legend,%
                title=$ \mathrm{I_{eff}} $, xlabel={\#dofs},
                log y ticks with fixed point,
                cycle list name=kanto,
            ]
            \addplot+[] table[x=dofs, y=Ieff, col sep=comma] {anc/moving-peak_4d_errors_o1_ni_amr+residual+doerfler-0.25.csv};
            \addplot+[] table[x=dofs, y=Ieff, col sep=comma] {anc/moving-peak_4d_errors_o2_ni_amr+residual+doerfler-0.25.csv};
            \addplot+[] table[x=dofs, y=Ieff, col sep=comma] {anc/moving-peak_4d_errors_o3_ni_amr+residual+doerfler-0.25.csv};
            \addplot+[] table[x=dofs, y=Ieff, col sep=comma] {anc/moving-peak_4d_errors_o4_ni_amr+residual+doerfler-0.25.csv};
            \addplot+[] table[x=dofs, y=Ieff, col sep=comma] {anc/moving-peak_4d_errors_o5_ni_amr+residual+doerfler-0.25.csv};
            \addplot+[MyBlue, mark=pentagon]    table[x=dofs, y=Ieff, col sep=comma] {anc/moving-peak_4d_errors_o1_ni_amr+functional+doerfler-0.25.csv};
            \addplot+[Viridian, mark=o]         table[x=dofs, y=Ieff, col sep=comma] {anc/moving-peak_4d_errors_o2_ni_amr+functional+doerfler-0.25.csv};
            \addplot+[Pewter]                   table[x=dofs, y=Ieff, col sep=comma] {anc/moving-peak_4d_errors_o3_ni_amr+functional+doerfler-0.25.csv};
            \addplot+[Cerulean, mark=x]         table[x=dofs, y=Ieff, col sep=comma] {anc/moving-peak_4d_errors_o4_ni_amr+functional+doerfler-0.25.csv};
            \addplot+[Vermilion]                table[x=dofs, y=Ieff, col sep=comma] {anc/moving-peak_4d_errors_o5_ni_amr+functional+doerfler-0.25.csv};
            \legend{{residual, $ p=1 $},
                    {residual, $ p=2 $},%
                    {residual, $ p=3 $},%
                    {residual, $ p=4 $},%
                    {residual, $ p=5 $},%
                    {functional, $ p=1 $},%
                    {functional, $ p=2 $},%
                    {functional, $ p=3 $},%
                    {functional, $ p=4 $},%
                    {functional, $ p=5 $},%
                    };
            \end{loglogaxis}
        \end{tikzpicture}
        }
    \end{subfigure}%
    \hfill%
    \begin{subfigure}{.425\linewidth}
        \includegraphics[width=\linewidth]{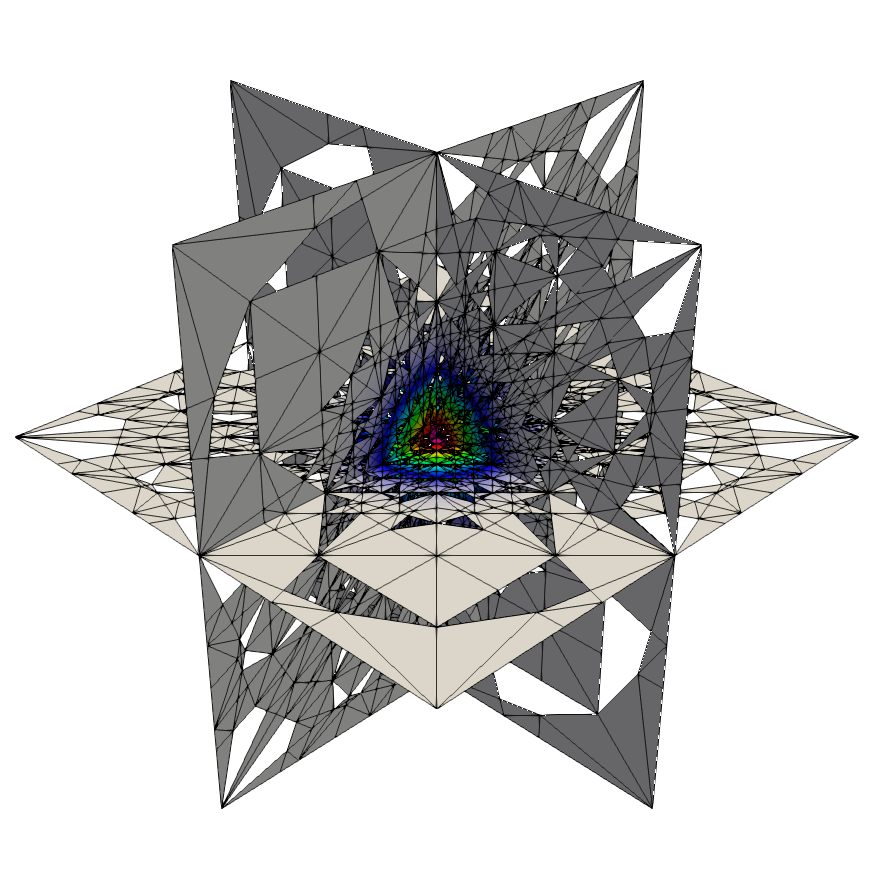}
    \end{subfigure}
    \caption{Example~\ref{LS:Subsection:MovingPeak}: Efficiency indices (lower left); and the mesh after 20 adaptive refinements, cut at $t=0.5$ (lower right); for $d=3$, and $\Xi = 0.25$.}%
    \label{LS:fig:moving-peak:efficiencyAndPlot}%
\end{figure}

%
%
\subsection{Circular Arc Scan Track}
\label{LS:Subsection:CircularArcScanTrack}
The next example was introduced in \cite{LS:CarraturoGiannelliRealiVazquez:2019a}
as a simplified model for simulating instationary heat transfer in additive manufacturing processes. 
We again have a very localized right-hand side, but this time we do not know the exact solution. 
Hence, we will compare the finite element solutions produced by adaptive refinement 
with a finite element solution obtained on a very fine, uniformly refined mesh. 
We consider the space-time cylinder $Q = (0,10)^{2}\times(0,5) $, 
the thermal conductivity coefficient 
$\nu \equiv 1$, 
the right hand side
\[
    f(x,t) = \num{2.97e5}\exp\!\left(-100\left((x_1-\bar{x}_1(t))^2 + (x_2-\bar{x}_2(t))^2\right)\right),
\]
with $(\bar{x}_1(t), \bar{x}_2(t)) = (5(1+\cos(\pi(5+2t)/{20})), 3 + 5\sin(\pi(5+2t)/{20}))$, 
the initial condition $u_0 \equiv 20$ on $\Sigma_0$,
and the Neumann boundary condition $ \nu\nabla_{x}u\cdot\vec{n}_x = 0$ on $\Sigma$;
see also \cite{LS:CarraturoGiannelliRealiVazquez:2019a}.
In the left plot of Figure~\ref{LS:fig:circular-arc:plots}, we present the convergence rates in the energy-norm $\|.\|_h$, for different combinations of polynomial degree $p$ and a posteriori error indicators/estimators. The stagnation at the end of each line is 
due to
the local mesh resolution of the adaptively refined mesh 
that
is getting smaller than the mesh resolution of the reference solution. In terms of convergence rates, the functional error estimator gives clearly better rates for linear elements, whereas the rates for quadratic and cubic elements are almost indistinguishable. However, when we compare the efficiency index $ \mathrm{I_{eff}} $, we observe a notable difference. While $ \mathrm{I_{eff}} \sim 1$ for the functional estimator independently of $p$, the efficiency index for the residual indicator is also close to $1$ for linear elements, but much worse 
when
we increase the polynomial degree.%
\begin{figure}[!htb]
    \centering%
    \begin{center}
        \ref{LS:leg:circular-arc}
    \end{center}%
    \begin{subfigure}{.5\linewidth}
        \resizebox*{\linewidth}{!}%
        {%
        \begin{tikzpicture}
            \begin{loglogaxis}[%
                    width=\linewidth,scale only axis,%
                    xmajorgrids,ymajorgrids,%
                    xlabel={\#dofs},%
                    cycle list name=adaptivity,
                    log y ticks with percent,
                    legend pos=south west,%
                    ticklabel style={font=\small},%
                    legend to name=LS:leg:circular-arc,%
                    legend columns=3,
                    log basis x=10,
                ]
                \addplot table [x=dofs, y=h-err, col sep=comma] {anc/circular_arc_scan_3d_errors_o1_ref-5_amr+functional+doerfler-0.25.csv};
                \addplot table [x=dofs, y=h-err, col sep=comma, restrict expr to domain={rawx}{1e3:1.5e5}] {anc/circular_arc_scan_3d_errors_o2_ref-4_amr+functional+doerfler-0.25.csv};
                \addplot table [x=dofs, y=h-err, col sep=comma, restrict expr to domain={rawx}{1e3:1.5e5}] {anc/circular_arc_scan_3d_errors_o3_ref-3_amr+functional+doerfler-0.25.csv};
                \addplot table [x=dofs, y=h-err, col sep=comma] {anc/circular_arc_scan_3d_errors_o1_ref-5_amr+residual+doerfler-0.25.csv};
                \addplot table [x=dofs, y=h-err, col sep=comma, restrict expr to domain={rawx}{1e3:1.5e5}] {anc/circular_arc_scan_3d_errors_o2_ref-4_amr+residual+doerfler-0.25.csv};
                \addplot table [x=dofs, y=h-err, col sep=comma, restrict expr to domain={rawx}{1e3:1.5e5}] {anc/circular_arc_scan_3d_errors_o3_ref-3_amr+residual+doerfler-0.25.csv};
                \addplot [darkgray,domain=2e4:8e4,densely dashed] {((x/1.1e5)^-(1./3))*.125};
                \addplot [darkgray,domain=2e4:5e4,densely dashdotted] {((x/40000)^-(2./3))*.08};
                \addplot [darkgray,domain=5e4:9e4,densely dashdotdotted] {((x/90000)^-1)*.035};
                \legend{
                    {functional, $ p=1 $},%
                    {functional, $ p=2 $},%
                    {functional, $ p=3 $},%
                    {residual, $ p=1 $},%
                    {residual, $ p=2 $},%
                    {residual, $ p=3 $},%
                }
            \end{loglogaxis}
        \end{tikzpicture}}
    \end{subfigure}%
    \hfill%
    \begin{subfigure}{.5\linewidth}
        \resizebox*{.925\linewidth}{!}%
        {%
        \begin{tikzpicture}
            \begin{loglogaxis}[%
                width=\linewidth,scale only axis,%
                xmajorgrids,ymajorgrids,%
                xlabel={\#dofs},%
                cycle list name=adaptivity,%
                log y ticks with fixed point,%
                legend columns=3,%
                transpose legend,%
                ticklabel style={font=\small},%
                ]
                \addplot table [x=dofs, y=Ieff, col sep=comma] {anc/circular_arc_scan_3d_errors_o1_ref-5_amr+functional+doerfler-0.25.csv};
                \addplot table [x=dofs, y=Ieff, col sep=comma, restrict expr to domain={rawx}{1e3:1.5e5}] {anc/circular_arc_scan_3d_errors_o2_ref-4_amr+functional+doerfler-0.25.csv};
                \addplot table [x=dofs, y=Ieff, col sep=comma, restrict expr to domain={rawx}{1e3:1.5e5}] {anc/circular_arc_scan_3d_errors_o3_ref-3_amr+functional+doerfler-0.25.csv};
                \addplot table [x=dofs, y=Ieff, col sep=comma] {anc/circular_arc_scan_3d_errors_o1_ref-5_amr+residual+doerfler-0.25.csv};
                \addplot table [x=dofs, y=Ieff, col sep=comma, restrict expr to domain={rawx}{1e3:1.5e5}] {anc/circular_arc_scan_3d_errors_o2_ref-4_amr+residual+doerfler-0.25.csv};
                \addplot table [x=dofs, y=Ieff, col sep=comma, restrict expr to domain={rawx}{1e3:1.5e5}] {anc/circular_arc_scan_3d_errors_o3_ref-3_amr+residual+doerfler-0.25.csv};
            \end{loglogaxis}
        \end{tikzpicture}}%
    \end{subfigure}%
    \caption{Example~\ref{LS:Subsection:CircularArcScanTrack}: Convergence rates in the energy norm (left); Efficiency indices $\mathrm{I_{eff}}$ (right); for different polynomial degrees $p$, and parameter $\Xi = 0.25$.}
    \label{LS:fig:circular-arc:plots}
\end{figure}
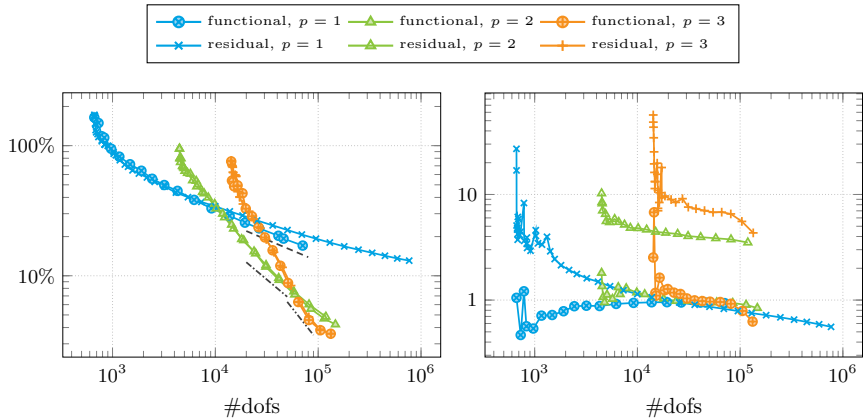
In Figure~\ref{LS:fig:circular-arc:meshes}, we present cuts through the space-time mesh, which was obtained after 10 adaptive refinements, using the functional estimator with bulk parameter $\Xi = 0.25 $. We can clearly observer that the mesh refinement follows the source. Moreover, the middle mesh seems to show triangles with very sharp angles, 
but which is due to the cut through the completely unstructured space-time mesh.%
\begin{figure}[!htb]
    \centering%
    \includegraphics[width=.32\linewidth]{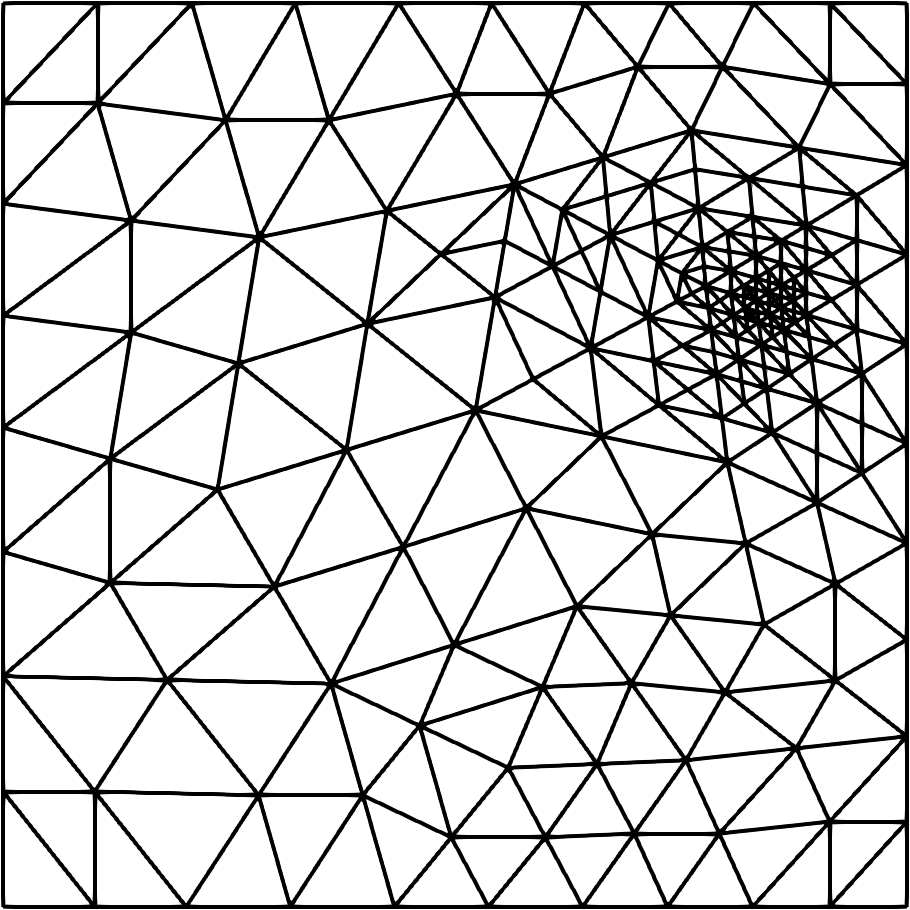}%
    \hfill%
    \includegraphics[width=.32\linewidth]{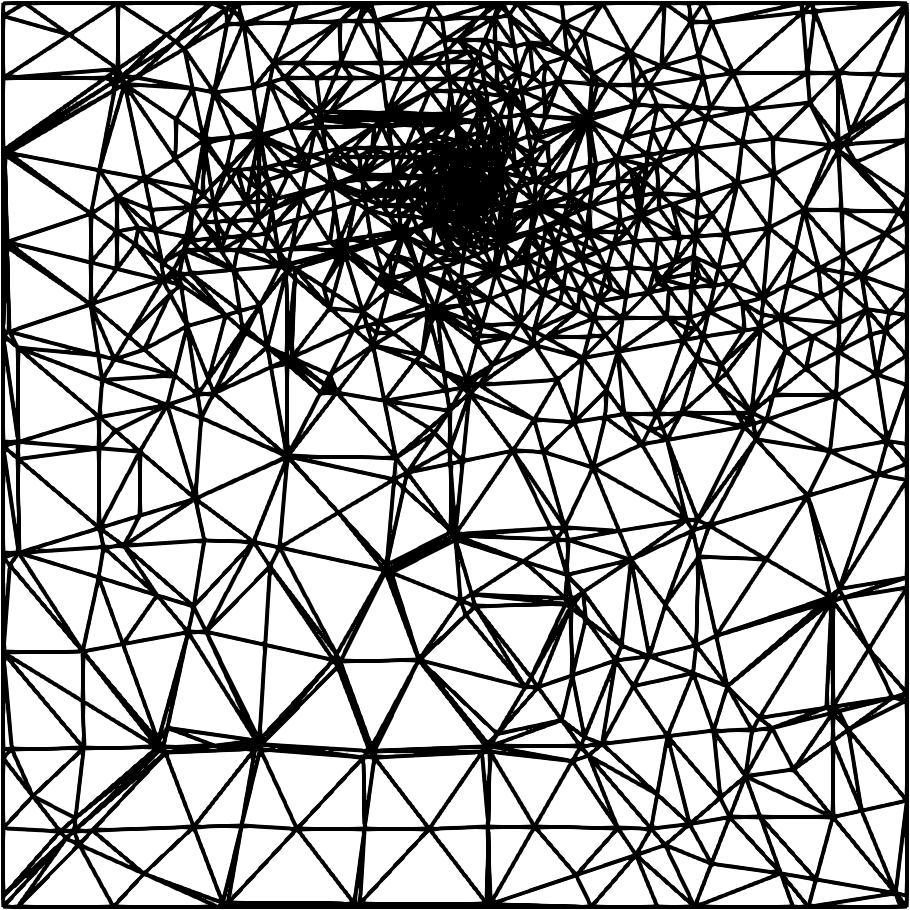}%
    \hfill%
    \includegraphics[width=.32\linewidth]{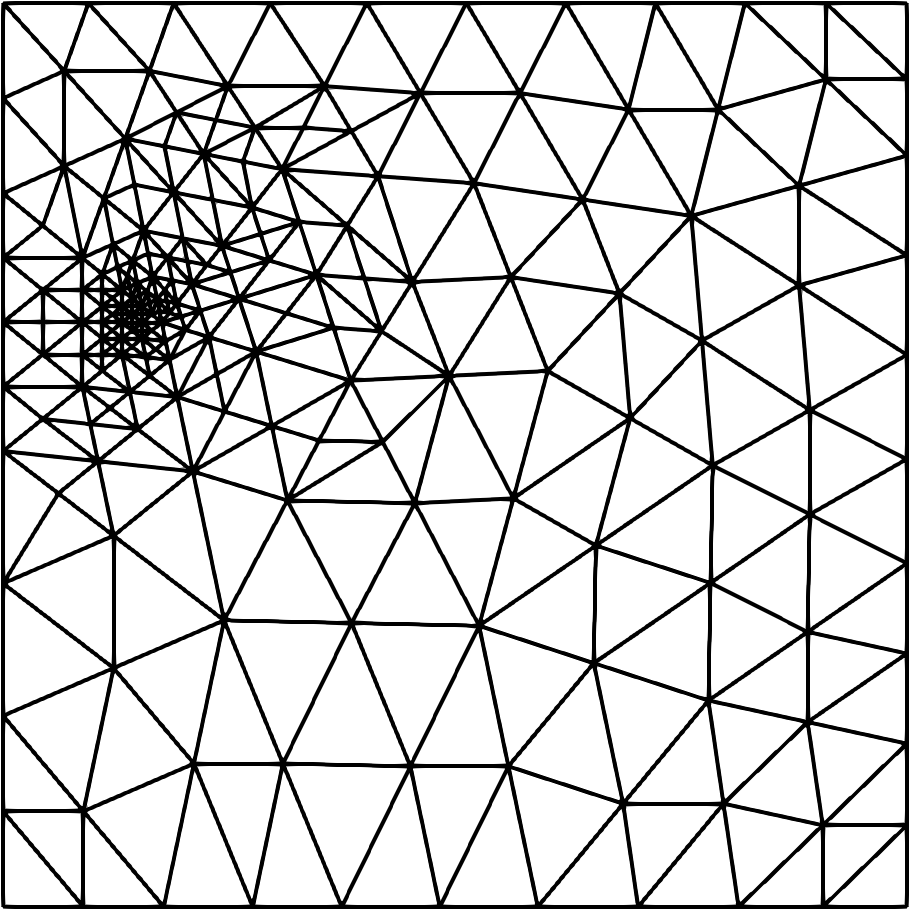}%
    \caption{Example~\ref{LS:Subsection:CircularArcScanTrack}: Cuts through the space-time mesh, after 10 adaptive refinements, from left to right, at times $t = 0$, $t=2.5$, and $t=5$.}
    \label{LS:fig:circular-arc:meshes}
\end{figure}
%

%
%
\subsection{Non-autonomous Kellogg's  interface problem}
\label{LS:Subsection:NonautonomousKellogg}
We choose the space-time cylinder $ Q = (-1,1)^2\times(0,1) $, and the discontinous
diffusion coefficient 
\[
    \nu(x,t) = \begin{cases}
        \nu_{13}(t),& (x,t)\in Q_1\cup Q_3,\\
        \nu_{24}(t), &(x,t)\in Q_2\cup Q_4,
    \end{cases}
\]
where $\nu_{13}(t) = R t + (1-t)$ and $\nu_{24}(t) = (1-t)/R + t$,  with $R = 161.4476387975884 $, and $\overline{Q} = \bigcup_{i=1}^{4}\overline{Q_i}$, $ Q_i = \Omega_i \times (0,1) $, $i = 1,\dots,4$, and $ \Omega_1 = (0,1)^2 $, $\Omega_2=(-1,0)\times(0,1)$, $\Omega_3=(-1,0)^2$ and $\Omega_4=(0,1)\times(-1,0)$. We use the manufactured solution
\[
    u(x,t) = r(x)^{\gamma}\ \mu(\varphi(x))\ t,
\]
with
\[
    \mu(\varphi) := \begin{cases}
        \cos((\frac{\pi}{2}-\sigma)\gamma) \cos((\varphi-\frac{\pi}{2} + \rho)\gamma), &0 \leq \varphi \leq \frac{\pi}{2},\\
        \cos(\gamma\rho)\cos((\varphi-\pi+\sigma)\gamma), & \frac{\pi}{2} \leq \varphi \leq \pi,\\
        \cos(\gamma\sigma)\cos((\varphi-\pi-\rho)\gamma), & \pi \leq \varphi \leq \frac{3\pi}{2},\\
        \cos(((\frac{\pi}{2})-\rho)\gamma)\cos((\varphi-3\frac{\pi}{2}-\sigma)\gamma), &\text{else},
    \end{cases}
\]
where $\gamma = 0.1$, $\rho = \pi/4  $, $\sigma = -19\pi/4 $, and $ (r,\varphi) $ represents the polar coordinates of the spatial variable $x$. The parameters $R$, $\rho$ and $\sigma$ are related via a nonlinear system of equations, which can be found in \cite{LS:Kellogg:1974a}, and determine the regularity of the function $u$. The above choice results in $ u \in H^{1+\gamma}(Q_i)$, i.e., $ u \in H^{1.1}(Q_i)$. According to the discretization error estimate \eqref{LS:eq:LocalInterpolationErrorEstimateSZ}, we cannot expect better convergence rates than $ \|u-u_h\|_h \sim \mathcal{O}(h^{\gamma}) $ for uniform mesh refinement, regardless of the chosen polynomial degree $p$. This behavior can be observed in Figure~\ref{LS:fig:kellogg-nonautonomous:convergence}, where the energy errors for uniform refinement have the same rate for $p=1,\dots,3$, i.e. $\sim \mathcal{O}(h^{0.1})$, up to a constant factor. When we compare the error rates for the adaptive refinements, we observe an improved rate of approximately $ \mathcal{O}(N_h^{-0.3/3}) $, which is still not (quasi-)optimal. One reason for this behavior might be the anisotropic nature of the singularity, which is located at the origin for all times $t$. Some preliminary tests 
with
hexahedral elements and anisotropic refinement indeed recovered the optimal rates, 
and the refinements were mostly concentrated in the spatial directions.
In the left part of Figure~\ref{LS:fig:kellogg-nonautonomous:efficiencyAndPlot}, we present the efficiency indices of the residual indicator and the functional estimator for different polynomial degrees $p$. In order to keep the plot readable, we cut off all values below 0.01 for the linear case of the residual incidator. For this example, the residual indicator yields efficiency indices $\mathrm{I_{eff}}\sim 1.5$ for $p=2$ and $\mathrm{I_{eff}}\sim 3$ for $p=3$, whereas the functional estimator gives efficiency indices $\mathrm{I_{eff}} \sim 10 - 20$, for all tested $p$. This is also resembled in the convergence rates, where we obtain slightly better results, in both value and rate, for the residual indicator with $p\ge2$; see Figure~\ref{LS:fig:kellogg-nonautonomous:convergence}.
\begin{figure}[!htb]
    \centering
    \begin{subfigure}{.85\linewidth}
    \resizebox*{\linewidth}{!}{
    \begin{tikzpicture}
        \begin{loglogaxis}[
                width=\linewidth,scale only axis,%
                xlabel={\#dofs},%
                title={$ \|u-u_h\|_h $},%
                log y ticks with fixed point,%
                cycle list name=adaptivity,%
                legend pos=north east,%
                legend columns=3, 
            ]
            \addplot table[x=dofs, y expr={\thisrow{h-err}*\thisrow{h-norm}}, col sep=comma] {anc/kellogg-0.1,non-autonomous_3d_errors_o1_uniform.csv};
            \addplot table[x=dofs, y expr={\thisrow{h-err}*\thisrow{h-norm}}, col sep=comma] {anc/kellogg-0.1,non-autonomous_3d_errors_o2_uniform.csv};
            \addplot table[x=dofs, y expr={\thisrow{h-err}*\thisrow{h-norm}}, col sep=comma] {anc/kellogg-0.1,non-autonomous_3d_errors_o3_uniform.csv};
            \addplot table[x=dofs, y=h-err, col sep=comma] {anc/kellogg-0.1,non-autonomous_3d_errors_o1_ni_amr+residual+doerfler-0.25.csv};
            \addplot table[x=dofs, y=h-err, col sep=comma] {anc/kellogg-0.1,non-autonomous_3d_errors_o2_ni_amr+residual+doerfler-0.25.csv};
            \addplot table[x=dofs, y=h-err, col sep=comma] {anc/kellogg-0.1,non-autonomous_3d_errors_o3_ni_amr+residual+doerfler-0.25.csv};
            \addplot+[Cerulean]    table[x=dofs, y=h-err-total, col sep=comma] {anc/kellogg-0.1,non-autonomous_3d_errors_o1_ni_amr+functional+doerfler-0.25.csv};
            \addplot+[LimeGreen]   table[x=dofs, y=h-err-total, col sep=comma] {anc/kellogg-0.1,non-autonomous_3d_errors_o2_ni_amr+functional+doerfler-0.25.csv};
            \addplot+[BurntOrange] table[x=dofs, y=h-err-total, col sep=comma] {anc/kellogg-0.1,non-autonomous_3d_errors_o3_ni_amr+functional+doerfler-0.25.csv};
            \addplot[domain=4276320.1:42763201,densely dashed, gray] {(x/42763201)^(-.1/3)*0.158};
            \label{LS:pgfplots:kellogg-nonautonomous:0.1}
            \addplot[domain=108907.6:1089076,densely dashdotted, gray] {(x/1089076)^(-.3/3)*0.09};
            \label{LS:pgfplots:kellogg-nonautonomous:0.3}
            \legend{
                {$p=1$, uniform},
                {$p=2$, uniform},
                {$p=3$, uniform},
                {$p=1$, residual},
                {$p=2$, residual},
                {$p=3$, residual},
                {$p=1$, functional},
                {$p=2$, functional},
                {$p=3$, functional},
                }
            \end{loglogaxis}
        \end{tikzpicture}}
    \end{subfigure}
    \caption{Example~\ref{LS:Subsection:NonautonomousKellogg}: Convergence rates for uniform and adaptive refinement, with bulk parameter $\Xi = 0.25$. The \ref{LS:pgfplots:kellogg-nonautonomous:0.1} line corresponds to a rate of $\mathcal{O}(N_h^{-0.1/3})$, and the \ref{LS:pgfplots:kellogg-nonautonomous:0.3} line to $\mathcal{O}(N_h^{-0.3/3})$.}
    \label{LS:fig:kellogg-nonautonomous:convergence}
\end{figure}
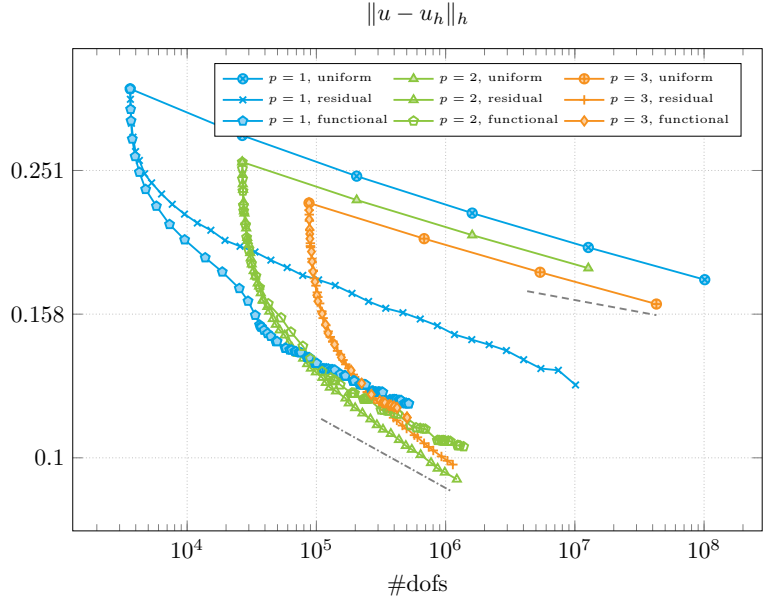
\begin{figure}[!htb]
    \centering
    \begin{subfigure}{.575\linewidth}
        \resizebox*{\linewidth}{!}{
        \begin{tikzpicture}
            \begin{loglogaxis}[%
                width=\linewidth,scale only axis,%
                legend pos=south east,
                legend style={font=\scriptsize},%
                title=$ \mathrm{I_{eff}} $, 
                xlabel={\#dofs},%
                log y ticks with fixed point,
                cycle list name=kanto,
            ]
            \addplot table[x=dofs, y=Ieff, col sep=comma, restrict expr to domain={rawy}{1e-2:100}] 
                                                          {anc/kellogg-0.1,non-autonomous_3d_errors_o1_ni_amr+residual+doerfler-0.25.csv};
            \addplot table[x=dofs, y=Ieff, col sep=comma] {anc/kellogg-0.1,non-autonomous_3d_errors_o2_ni_amr+residual+doerfler-0.25.csv};
            \addplot table[x=dofs, y=Ieff, col sep=comma] {anc/kellogg-0.1,non-autonomous_3d_errors_o3_ni_amr+residual+doerfler-0.25.csv};
            \addplot+[MyBlue]
                     table[x=dofs, y=Ieff, col sep=comma] {anc/kellogg-0.1,non-autonomous_3d_errors_o1_ni_amr+functional+doerfler-0.25.csv};
            \addplot+[Viridian]
                     table[x=dofs, y=Ieff, col sep=comma] {anc/kellogg-0.1,non-autonomous_3d_errors_o2_ni_amr+functional+doerfler-0.25.csv};
            \addplot+[Pewter]
                     table[x=dofs, y=Ieff, col sep=comma] {anc/kellogg-0.1,non-autonomous_3d_errors_o3_ni_amr+functional+doerfler-0.25.csv};
            \legend{%
                    {residual, $ p=1 $},
                    {residual, $ p=2 $},%
                    {residual, $ p=3 $},%
                    {functional, $ p=1 $},%
                    {functional, $ p=2 $},%
                    {functional, $ p=3 $},%
                    };
            \end{loglogaxis}
        \end{tikzpicture}}
    \end{subfigure}%
    \hfill%
    \begin{subfigure}{.425\linewidth}
        \centering
        \includegraphics[width=.9\linewidth]{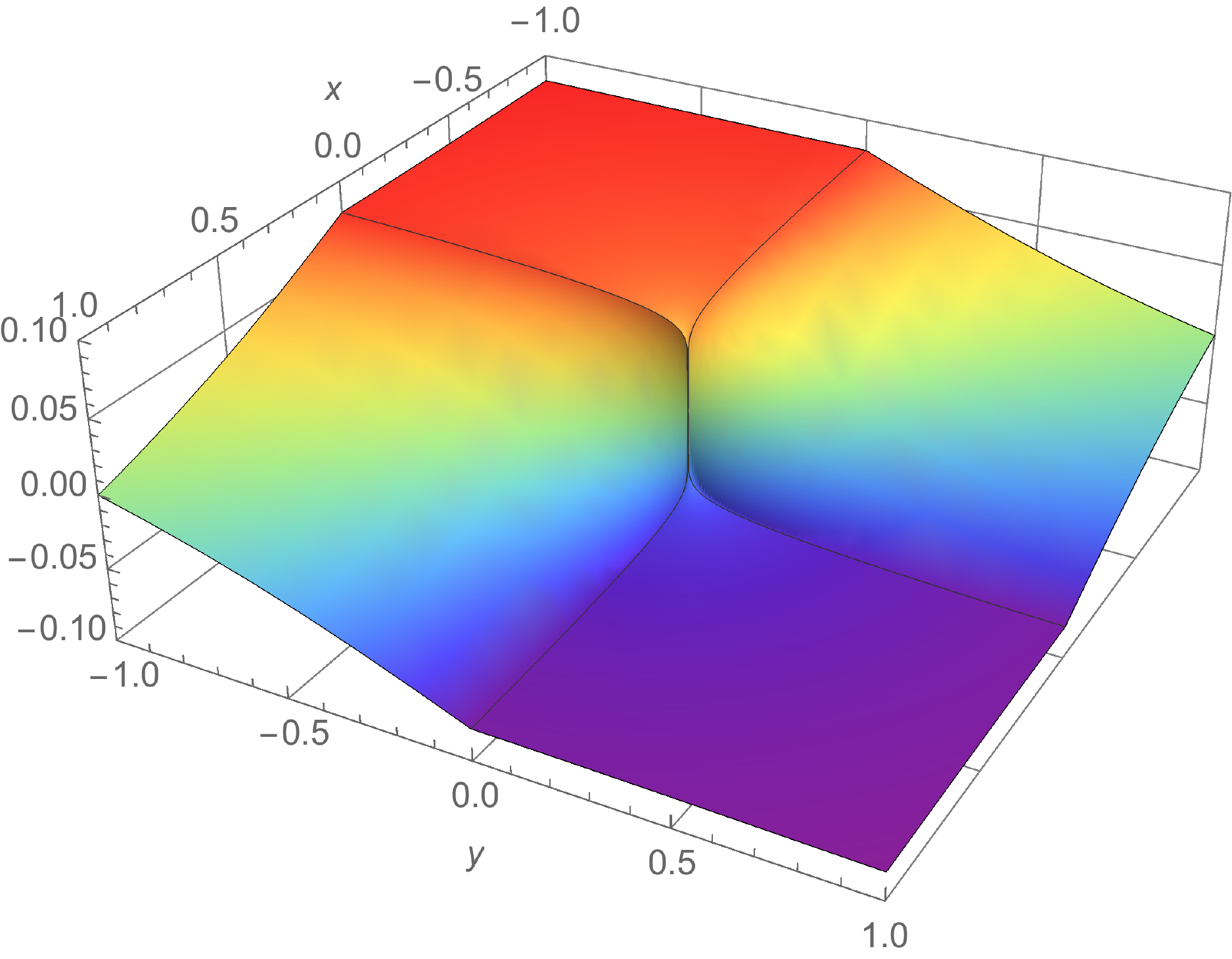}
        \vspace{2em}%
    \end{subfigure}
    \caption{Example~\ref{LS:Subsection:NonautonomousKellogg}: Efficiency indices (left); and plot of the solution at $t=1$ (right); with marking threshold $\Xi = 0.25$.}%
    \label{LS:fig:kellogg-nonautonomous:efficiencyAndPlot}
\end{figure}
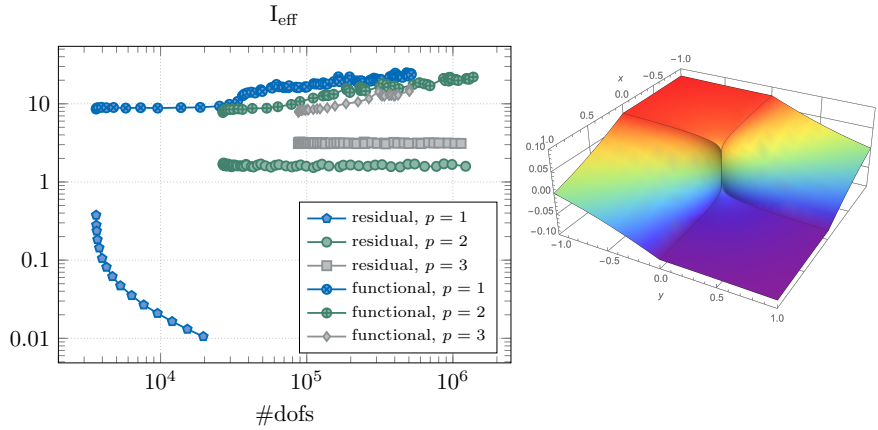

%
%
\section{Conclusions}
\label{LS:Section:Conclusions}
We presented and analyzed locally stabilized, consistent, conforming finite element schemes 
on completely unstructured simplicial space-time meshes for non-autonomous parabolic initial-boundary value problems.
We admitted special distributional right-hand sides and (diffusion) coefficients 
which can be discontinuous in space and time. 
Such data can lead to low-regularity solutions.
We derived a priori estimates for solutions belonging to $H^k(Q)$ with some $k \in  (1,2]$.
In this case, uniform mesh refinement leads to low convergence rates. 
More precisely, in the energy norm $\| \cdot \|_h$, we get $O(h^{k-1})$ independently of 
the polynomial degree $p$ of the shape functions used on the reference element.
This theoretical result is confirmed by all numerical experiments performed.
The convergence rate can drastically be improved by adaptivity.
In order to devise an adaptive finite element scheme, one needs an local error indicator
derived from a posteriori discretization error estimators. 
In contrast to elliptic boundary-value problems where adaptive finite element
schemes are well established in theory and practice, the picture is different 
with respect to fully adaptive space-time finite element methods.
We numerically studied the residual indicator proposed by O. Steinbach and H. Yang,
and indicators based on functional a posteriori estimators proposed by S. Repin, 
where we only used the first part of the majorant as 
local error indicator.
This choice already yields 
an indicator
that 
is reliable and provides an upper bound with 
efficiency indicies that are close to 1 
for the first two examples, whereas the residual indicator is only reliable for higher polynomial degrees, with efficiency indices much bigger than 1. 
However, for the third example, the residual indicator results in much better efficiency indices than the functional estimator/indicator. 
The reduced performance of the latter is most likely due to two reasons: 
first, we only use 
the continuous finite element spaces $(V_h)^d$ for flux recovering that do not reflect the right behaviour 
of the fluxes across material interfaces that arise in the third example,
and second, the singularity is highly anisotropic in space and time, i.e., we would need a very high spatial resolution, but relatively coarse one in time. 
Anisotropic refinement techniques, in particular, on simplicial space-time meshes is certainly a challenge
that will be a topic of our future research on space-time adaptivity.

\section{Acknowledgment}
The authors would like to thank the Austrian Science Fund  (FWF) 
for the financial support
under the grant DK W1214-04. Furthermore, we would like to thank 
Sergey Repin for discussing with us the use of functional error 
estimates during his visits at Linz.
Andreas Schafelner wants to thank Panayot Vassilevski for the
support during his visits at the Lawrence Livermore National Laboratory,
for many fruitful discussions, and for the possibility to compute on
the  distributed memory cluster Quartz
in Livermore.

\bibliographystyle{acm}
\bibliography{ms}

\end{document}